\documentclass[a4paper,11pt,twoside,notitlepage]{amsart}

\usepackage{amsmath,amssymb,amsthm,mathtools}
\usepackage{xcolor}

\usepackage[ocgcolorlinks,linkcolor=blue,citecolor=blue,urlcolor=blue]{hyperref}
\usepackage{url}
\numberwithin{equation}{section}
\mathtoolsset{showonlyrefs}
\allowdisplaybreaks

\newtheorem{theorem}{Theorem}[section]
\newtheorem{proposition}[theorem]{Proposition}
\newtheorem{lemma}[theorem]{Lemma}
\newtheorem{corollary}[theorem]{Corollary}
\theoremstyle{remark}
\newtheorem{remark}[theorem]{Remark}

\DeclareMathOperator{\Diff}{Diff}
\DeclareMathOperator{\Id}{Id}
\DeclareMathOperator{\supp}{supp}
\DeclareMathOperator{\diver}{div}
\newcommand{\p}{\partial}
\newcommand{\ol}{\overline}
\newcommand{\wh}{\widehat}
\newcommand{\R}{\mathbb R}
\newcommand{\C}{\mathbb C}
\newcommand{\para}[1]{\vspace{3mm}\noindent\textbf{#1.}}
\newcommand{\Hs}{H^1_{\mathrm{scl}}}
\newcommand{\Hsm}{H^{-1}_{\mathrm{scl}}}

\title[Near-Euclidean anisotropic Calder\'on rigidity]{The anisotropic Calder\'on problem: rigidity near the Euclidean metric}

\author[Y.-H. Lin]{Yi-Hsuan Lin}
\address{Department of Applied Mathematics, National Yang Ming Chiao Tung University, Hsinchu, Taiwan; Fakult\"at f\"ur Mathematik, University of Duisburg-Essen, Essen, Germany}
\email{yihsuanlin3@gmail.com}


\hypersetup{
	pdftitle={The anisotropic Calder\'on problem: rigidity near the Euclidean metric},
	pdfauthor={Yi-Hsuan Lin},
	pdfsubject={Near-Euclidean rigidity for the anisotropic Calderon problem},
	pdfkeywords={anisotropic Calderon problem, near-Euclidean metric, Dirichlet-to-Neumann map, harmonic coordinates, Carleman estimate, rigidity}
}

\begin{document}
	
\begin{abstract}
	We prove rigidity near the Euclidean metric for the anisotropic Calder\'on problem on smooth bounded connected domains in dimensions three and higher. Any smooth Riemannian metric sufficiently close to the Euclidean metric in a suitable H\"older norm and having the same Dirichlet-to-Neumann map agrees with it up to a diffeomorphism fixing the boundary pointwise. The result applies to general smooth anisotropic perturbations, without analytic or quasianalytic regularity assumptions, a prescribed conformal class, or a transversal product structure, and requires no convexity of the boundary. The proof combines harmonic coordinates with a Carleman estimate for compactly supported correctors to obtain a quadratic Fourier estimate on a frequency range determined by the perturbation. A frequency decomposition then yields rigidity.
\end{abstract}
	
	\maketitle
	\tableofcontents
	
	\section{Introduction}
	
	Inverse boundary value problems ask whether an unknown medium can be determined from measurements made at its boundary. Let $\Omega\subset\R^n$ be a smooth bounded domain and let $A=(A^{ij})$ be a real symmetric uniformly elliptic matrix-valued function. For a prescribed boundary value $f$, let $u$ solve
	\begin{equation}\label{eq:intro-conductivity}
		\begin{cases}
			-\diver(A\nabla u)=0 & \text{in }\Omega,\\
			u=f & \text{on }\p\Omega.
		\end{cases}
	\end{equation}
	The associated Dirichlet-to-Neumann (DN) map is $\Lambda_Af=A\nabla u\cdot\nu|_{\p\Omega}$, where $\nu$ is the outward unit normal. The problem was introduced in \cite{Calderon1980}. For isotropic conductivities $A=\gamma I$, global uniqueness for sufficiently smooth $\gamma$ in dimensions $n\ge3$ was proved in \cite{SylvesterUhlmann1987}; see \cite{FeldmanSaloUhlmann2025} for an introduction and further references.
	
	For anisotropic conductivities, uniqueness in fixed coordinates is impossible. If $F:\ol\Omega\to\ol\Omega$ is a smooth diffeomorphism satisfying $F|_{\p\Omega}=\Id$, define $(F_*A)(y)=DF(x)A(x)DF(x)^T/|\det DF(x)|$, where $x=F^{-1}(y)$, then $\Lambda_{F_*A}=\Lambda_A$. In dimensions $n\ge3$, the conductivity equation is equivalent to a Laplace--Beltrami equation through $A_g=(\det g)^{1/2}g^{-1}$ and $g_A=(\det A)^{1/(n-2)}A^{-1}$. The anisotropic Calder\'on problem asks whether a boundary fixing change of coordinates is the only obstruction to recovering a Riemannian metric from its DN map. In two dimensions, quasiconformal changes of coordinates reduce anisotropic conductivities to isotropic ones. Global uniqueness up to a boundary fixing quasiconformal change of coordinates holds for bounded measurable uniformly elliptic conductivities on simply connected planar domains \cite{AstalaLassasPaivarinta2005}; see \cite[Section~5.11]{FeldmanSaloUhlmann2025} for further discussion.
	
	For smooth Riemannian metrics in dimensions $n\ge3$, the DN map determines the induced boundary metric and the full boundary jet in boundary normal coordinates \cite{LeeUhlmann1989,JoshiLionheart2005}. Global uniqueness for real analytic metrics was first proved under additional geometric and topological assumptions in \cite{LeeUhlmann1989}. These additional assumptions were removed in \cite{LassasUhlmann2001}. The Poisson embedding method gives another proof of analytic rigidity and a global construction of the boundary fixing isometry \cite{LassasLiimatainenSalo2020}. For smooth metrics, uniqueness in a prescribed conformal class is known on conformally transversally anisotropic manifolds under suitable assumptions on the transversal geometry \cite{DosSantosFerreiraKenigSaloUhlmann2009,DosSantosFerreiraKurylevLassasSalo2016}. Recent work proves rigidity for quasianalytic metrics and for certain smooth metrics with symmetry or one-sided ordering \cite{ChenJiangLiuTao2026}. The problem for general smooth metrics in dimensions $n\ge3$ remains open \cite{FeldmanSaloUhlmann2025,ChenJiangLiuTao2026}.
	
	Nonuniqueness is known in several related settings. Counterexamples for partial boundary data and metrics with limited regularity are constructed in \cite{DaudeKamranNicoleau2020PartialData}. More recent examples give smooth nonisometric coefficients with identical full boundary data for equations with prescribed nonconstant potentials or nonzero spectral terms \cite{DaudeEncisoHelfferKamranNicoleau2026}. These examples concern different data or equations from \eqref{eq:intro-conductivity} with full boundary measurements.
	
	Singular transformation constructions give another form of nonuniqueness: a singular anisotropic metric can have the same boundary measurements as the Euclidean metric without arising from a smooth boundary fixing change of coordinates \cite{GreenleafKurylevLassasUhlmann2009}. The theorem below shows that this type of invisibility beyond the diffeomorphism gauge cannot occur for smooth metrics sufficiently close to $e$.
	
	Related results have been obtained for Lorentzian wave equations. Zeroth-order coefficients can be recovered under curvature assumptions and near the Minkowski geometry \cite{AlexakisFeizmohammadiOksanen2022,AlexakisFeizmohammadiOksanen2025}. For globally hyperbolic Lorentzian metrics that agree with the Minkowski metric outside a compact set, equality of the hyperbolic DN map with the Minkowski one implies rigidity up to a boundary fixing diffeomorphism \cite{OksanenRakeshSalo2026Rigidity}. This result follows from a stronger theorem using a formally determined collection of distorted plane wave measurements. For time-independent hyperbolic operators with the same lower-order terms, finitely many plane waves and complementary solutions determine the leading coefficients under geometric assumptions on one of the operators \cite{OksanenRakeshSalo2026FixedAngle}. Semiglobal uniqueness for Lorentzian metrics is also obtained when one of the two metrics is sufficiently close to the Minkowski metric \cite{OksanenRakeshSalo2026}. These results use hyperbolic boundary data, while the present paper concerns the elliptic DN map.
	
	A recent near-Euclidean result determines compactly supported scalar potentials on a fixed known three-dimensional metric that is $C^4$ close to the Euclidean metric and has strictly convex boundary, with a consequence for conformal multiples of the fixed metric \cite{UhlmannWang2026NearEuclidean}. The present paper proves nonlinear rigidity at the Euclidean metric for a general smooth anisotropic perturbation of the principal coefficient. The equality $\Lambda_g=\Lambda_e$ forces the metric to be Euclidean up to a boundary fixing diffeomorphism. We assume no convexity of the boundary, prescribed conformal class, transversal product structure, or analytic or quasianalytic regularity. The result holds in every dimension $n\ge3$, with smallness measured in $C^{k,\alpha}$, where $k=\lfloor n/2\rfloor+1$ and $n/2+1-k<\alpha<1$. Smoothness is used to recover and compare the full boundary jets, while the quantitative argument uses only $\|g-e\|_{C^{k,\alpha}(\ol\Omega)}$. Throughout the paper, $e$ denotes the Euclidean metric. A closely related rigidity theorem for smooth metrics near the Euclidean metric is proved in \cite{Stefanov2026}, with smallness measured in $H^s$ for an integer $s>n/2+1$.

	\para{Main results} We use the convention $\Delta_g=\operatorname{div}_g\nabla_g$, so that $-\Delta_g$ is nonnegative. For a smooth Riemannian metric $g$ on $\ol\Omega$, the DN map is
	\begin{equation*}
		\Lambda_g:C^\infty(\p\Omega)\to C^\infty(\p\Omega),\quad f\mapsto \p_{\nu_g}u\big|_{\p\Omega},
	\end{equation*}
	where $u$ is the unique solution of
	\begin{equation*}
		\begin{cases}
			-\Delta_gu=0 & \text{in }\Omega,\\
			u=f & \text{on }\p\Omega,
		\end{cases}
	\end{equation*}
	and $\nu_g$ is the unit outer normal with respect to $g$. We write $\Diff^\infty(\ol\Omega)$ for the smooth diffeomorphisms of $\ol\Omega$ onto itself, smooth up to the boundary, and $\Id$ for the identity map. We write $\operatorname{Sym}^2\R^n$ for the space of real symmetric $n\times n$ matrices and denote the Euclidean variable by $x=(x^1,\ldots,x^n)$. Repeated coordinate indices are summed from $1$ to $n$.

	Set $s_*=\frac n2+1$ and $k=\lfloor s_*\rfloor$. The H\"older exponent in the metric statements satisfies $s_*-k<\alpha<1$, so $k+\alpha>s_*$. When $n$ is even, this condition is $0<\alpha<1$; when $n$ is odd, it is $1/2<\alpha<1$.
	
	\begin{theorem}[Rigidity near the Euclidean metric]\label{thm:metric-main}
		Let $\Omega\subset\R^n$ be a smooth bounded connected domain, where $n\ge3$, and let $s_*-k<\alpha<1$. There exists $\varepsilon_1=\varepsilon_1(n,\Omega,\alpha)>0$ with the following property. Suppose that $g\in C^\infty(\ol\Omega;\operatorname{Sym}^2\R^n)$ is a Riemannian metric satisfying
		\begin{equation}\label{eq:metric-smallness}
			\|g-e\|_{C^{k,\alpha}(\ol\Omega)}<\varepsilon_1.
		\end{equation}
		Assume that
		\begin{equation}\label{eq:same-DN-map}
			\Lambda_g=\Lambda_e.
		\end{equation}
		For $1\le j\le n$, let $y^j$ solve
		\begin{equation}\label{eq:BVP-harmonic-coord}
			\begin{cases}
				-\Delta_gy^j=0 & \text{in }\Omega,\\
				y^j=x^j & \text{on }\p\Omega
			\end{cases}
		\end{equation}
		and set $Y=(y^1,\ldots,y^n)$. Then $Y\in\Diff^\infty(\ol\Omega)$, $Y|_{\p\Omega}=\Id$, and
		\begin{equation}\label{eq:metric-rigidity-conclusion}
			g=Y^*e.
		\end{equation}
		Moreover, $Y$ is the unique diffeomorphism $F\in\Diff^\infty(\ol\Omega)$ satisfying $F|_{\p\Omega}=\Id$ and $g=F^*e$.
	\end{theorem}
	
	The conclusion is the natural one because every boundary fixing diffeomorphism preserves the DN map.
	
	\para{Harmonic coordinate reduction and the conductivity theorem}
	A related harmonic coordinate argument for the linearized anisotropic Calder\'on problem appears in \cite{SylvesterUhlmann1991}; see also \cite[Section~9, pp.~190--193]{Uhlmann1992}. We use the harmonic map determined by the boundary values of the Cartesian coordinate functions to fix the diffeomorphism freedom in the nonlinear problem. Let $y^j$ solve \eqref{eq:BVP-harmonic-coord}, and set $Y=(y^1,\ldots,y^n)$. When $g$ is sufficiently close to $e$, this map is a boundary fixing diffeomorphism. If $\widetilde g=Y_*g$ and $\widetilde A=(\det\widetilde g)^{1/2}\widetilde g^{-1}$, then the coordinate functions are $\widetilde g$-harmonic and
	\begin{equation}\label{eq:intro-harmonic-gauge}
		\p_i\widetilde A^{ij}=0
		\quad\text{for }1\le j\le n.
	\end{equation}
	Writing $\widetilde A=I+H$, \eqref{eq:intro-harmonic-gauge} gives $\diver H=0$. This is an exact condition in the fixed Euclidean target coordinates, not a linearized gauge condition. It follows from the choice of coordinates and is not an additional assumption on the original metric in Theorem~\ref{thm:metric-main}. Moreover, Proposition~\ref{prop:harmonic-coordinates} gives $\|H\|_{C^{k,\alpha}(\ol\Omega)}\le C\|g-e\|_{C^{k,\alpha}(\ol\Omega)}$, so the transformed conductivity remains close to $I$ in the same H\"older norm.
	
	The key result used after this normalization is the following conductivity theorem. The compact support condition below is an assumption of the conductivity theorem and does not follow directly from the harmonic coordinate normalization. For a matrix field $H=(H^{ij})$, the notation $\diver H=0$ means $\p_iH^{ij}=0$ in distributions for every $j$.
	
	\begin{theorem}[Rigidity in harmonic gauge]\label{thm:conductivity-main}
		Let $\Omega\subset\R^n$ be a smooth bounded connected domain, where $n\ge3$, and let $K\Subset\Omega$ be fixed. There exists $\varepsilon_0>0$, depending only on $n$, $\Omega$, and $K$, with the following property. Let $A=I+H$ be a real symmetric uniformly elliptic conductivity such that
		\begin{equation}\label{eq:main-H-assumptions}
			H\in H^{s_*}(\R^n;\operatorname{Sym}^2\R^n),\quad \supp H\subset K,\quad \diver H=0.
		\end{equation}
		If $\Lambda_A=\Lambda_I$ and $\|H\|_{H^{s_*}(\R^n)}<\varepsilon_0$, then $H=0$.
	\end{theorem}
	
	For a complex matrix $M=(M_{jk})\in\C^{n\times n}$, we write $|M|=(\sum_{j,k=1}^n|M_{jk}|^2)^{1/2}$. The $L^2$ and Sobolev norms of matrix fields are formed using this pointwise matrix norm. Throughout the proof, the matrix $L^\infty$ norm is the operator norm
	$\|H\|_{L^\infty}:=\operatorname*{ess\,sup}_{x\in\R^n}\sup_{v\in\C^n,\ |v|=1}|H(x)v|$.
	
	Theorem~\ref{thm:conductivity-main} requires smallness in $H^{s_*}(\R^n)$, where $s_*=n/2+1$. When $n=3$, this is the space $H^{5/2}(\R^3)$. The metric theorem uses $C^{k,\alpha}$ smallness with $k=\lfloor s_*\rfloor$ and $\alpha>s_*-k$, which becomes $C^{2,\alpha}$ with $\alpha>1/2$ in dimension three. The theorem is not applied directly to the perturbation $H=A_{Y_*g}-I$ on the original domain, since the harmonic coordinate normalization gives $\diver H=0$ but does not imply $\supp H\Subset\Omega$. Smooth boundary determination gives an auxiliary boundary fixing gauge in which the metric has the same full boundary jet as $e$. Comparing this gauge with the canonical harmonic gauge shows that $H$ vanishes to infinite order at the boundary. The perturbation can then be extended by zero to a fixed larger domain. The condition $k+\alpha>s_*$ gives an $H^{s_*}(\R^n)$ bound for this extension in terms of $\|H\|_{C^{k,\alpha}(\ol\Omega)}$, and Theorem~\ref{thm:conductivity-main} applies.

	\para{Complex null vectors} We write $\mathsf{i}=\sqrt{-1}$ and extend the Euclidean dot product on $\R^n$ complex bilinearly to $\C^n$, so that $\zeta\cdot\eta=\sum_{j=1}^n\zeta_j\eta_j$ for $\zeta,\eta\in\C^n$ (without an additional complex conjugate). A vector $\zeta\in\C^n$ is called a complex null vector if $\zeta\cdot\zeta=0$. The notation $|\zeta|=(\sum_{j=1}^n|\zeta_j|^2)^{1/2}$ denotes the usual Hermitian norm on $\C^n$. If $\zeta=\alpha+\mathsf{i}\beta$ with $\alpha,\beta\in\R^n$, then the null condition is equivalent to $\alpha\cdot\beta=0$ and $|\alpha|=|\beta|$. Since $\Delta e^{\zeta\cdot x}=(\zeta\cdot\zeta)e^{\zeta\cdot x}$, the function $e^{\zeta\cdot x}$ is harmonic whenever $\zeta$ is a complex null vector.
	
	\para{Outline of the proof} We first reduce Theorem~\ref{thm:metric-main} to Theorem~\ref{thm:conductivity-main}. The canonical harmonic map gives a boundary fixing diffeomorphism $Y$ and a transformed conductivity $I+H$ satisfying $\diver H=0$, together with the estimate $\|H\|_{C^{k,\alpha}(\ol\Omega)}\le C\|g-e\|_{C^{k,\alpha}(\ol\Omega)}$. Smooth boundary determination supplies a second boundary fixing diffeomorphism for which the metric has the same full boundary jet as $e$. Comparing the two gauges shows that the metric in the canonical harmonic coordinates, and hence $H$, is flat at the boundary. Extending $H$ by zero to a fixed larger domain produces a smooth compactly supported divergence free tensor. Since $k+\alpha>s_*$, the $H^{s_*}$ norm of this extension is controlled by $\|H\|_{C^{k,\alpha}(\ol\Omega)}$. A weak gluing argument transfers equality of the DN maps to the outer boundary. Theorem~\ref{thm:conductivity-main} then gives $H=0$, so the metric is Euclidean in the harmonic coordinates.
	
	For the proof of Theorem~\ref{thm:conductivity-main}, set $X=\|H\|_{L^2}$, $\delta=\|H\|_{L^\infty}$, and $M=\|H\|_{H^{s_*}}$. The difficulty is that $H$ changes the principal part of the equation. Conjugation by a harmonic exponential produces terms quadratic in the complex frequency, so the perturbation estimate used here cannot be absorbed at arbitrarily large frequencies for a fixed nonzero $H$. We instead work on a finite frequency range determined by $\delta$ and compare the resulting estimate with the Sobolev tail of the same tensor $H$.
	
	For a complex null vector $\zeta$, let $u_\zeta$ be the $A$-harmonic solution with boundary value $e^{\zeta\cdot x}$. Equality of the DN maps gives a compactly supported corrector $r_\zeta$ such that $u_\zeta=e^{\zeta\cdot x}(1+r_\zeta)$. A Carleman estimate for compactly supported functions gives
	\begin{equation}\label{eq:intro-corrector-estimate}
		\|(\nabla+\zeta)r_\zeta\|_{L^2}\le C|\zeta|^2X,\quad\text{for }1\lesssim|\zeta|\lesssim\delta^{-1}.
	\end{equation}
	Pairing $u_\zeta$ with an explicit harmonic exponential, using $\diver H=0$, and estimating the integral remainder by the Cauchy--Schwarz inequality yields
	\begin{equation}\label{eq:intro-Fourier-estimate}
		|\wh H(\xi)|\le C(1+|\xi|)X^2,\quad\text{for }|\xi|\lesssim\delta^{-1}.
	\end{equation}
	The low-frequency part is bounded by $CR^{s_*}X^2$, while the $H^{s_*}$ norm bounds the remaining tail by $CR^{-s_*}M$. For $X>0$, choose $R=\kappa X^{-1/s_*}$ with a fixed sufficiently small $\kappa>0$. Gagliardo--Nirenberg interpolation gives $\delta\le CX^{1/s_*}M^{1-1/s_*}$, so $R\delta\le C\kappa M^{1-1/s_*}$ and the chosen radius lies in the available frequency range when $M$ is sufficiently small. Absorbing the low-frequency term gives $X\le CMX$, and smallness of $M$ forces $X=0$.
	
	\para{Organization of the paper} In Section~\ref{sec:preliminaries}, we prove the mixed comparison identity, construct the compactly supported correctors, and derive their conjugated equation. The Carleman estimate and the resulting corrector bounds are established in Section~\ref{sec:carleman}. We then use these bounds to obtain the Fourier estimate and prove Theorem~\ref{thm:conductivity-main} in Section~\ref{sec:fourier}. Section~\ref{sec:metric} returns to the metric problem. There we construct the canonical harmonic gauge and compare it with the boundary determination. An exterior gluing argument transfers equality of the DN maps to a larger domain, allowing us to apply the conductivity theorem and complete the proof of Theorem~\ref{thm:metric-main}.

	\section{Exterior normalization and the conjugated equation}\label{sec:preliminaries}
	
	\para{The DN map and the mixed identity} Let $A$ be a real symmetric uniformly elliptic conductivity on $\Omega$. All complex-valued dualities and energy pairings below are understood in the complex bilinear sense, with no complex conjugation, unless a conjugate is displayed explicitly. For $f\in H^{1/2}(\p\Omega)$, let $u_f\in H^1(\Omega)$ be the unique weak solution of \eqref{eq:intro-conductivity}. The DN map is defined weakly by
	\begin{equation}\label{eq:weak-DN}
		\langle\Lambda_Af,g\rangle  =\int_\Omega A\nabla u_f\cdot\nabla v\,dx,
	\end{equation}
	where $v\in H^1(\Omega)$ is any function whose boundary trace is $v|_{\p\Omega}=g\in H^{1/2}(\p\Omega)$. The definition is independent of the extension $v$. The following is the comparison identity used throughout the proof.
	
	\begin{lemma}[Mixed comparison identity]\label{lem:mixed-identity}
		Suppose that $A=I+H$ and $\Lambda_A=\Lambda_I$. Let $u\in H^1(\Omega)$ satisfy $-\diver(A\nabla u)=0$, and let $v\in H^1(\Omega)$ satisfy $-\Delta v=0$. Then
		\begin{equation}\label{eq:mixed-identity}
			\int_\Omega H\nabla u\cdot\nabla v\,dx=0.
		\end{equation}
	\end{lemma}
	
	\begin{proof}
		Let $f=u|_{\p\Omega}$ and $g=v|_{\p\Omega}$. By the weak definitions \eqref{eq:weak-DN} of the two DN maps,
		\begin{equation}\label{eq:mixed-proof-1}
			\int_\Omega A\nabla u\cdot\nabla v\,dx
			=\langle\Lambda_Af,g\rangle
			=\langle\Lambda_If,g\rangle.
		\end{equation}
		Let $u_0$ be the harmonic function with trace $f$. Since $v$ is Euclidean harmonic,
		\begin{equation}\label{eq:mixed-proof-2}
			\langle\Lambda_If,g\rangle
			=\int_\Omega\nabla u_0\cdot\nabla v\,dx
			=\int_\Omega\nabla u\cdot\nabla v\,dx.
		\end{equation}
		The last equality follows because $u-u_0\in H_0^1(\Omega)$ and $v$ is harmonic. Subtracting \eqref{eq:mixed-proof-2} from \eqref{eq:mixed-proof-1} proves \eqref{eq:mixed-identity}.
	\end{proof}
	
	\begin{remark}\label{rem:mikko-comparison}
		Lemma~\ref{lem:mixed-identity} uses only one solution for the unknown conductivity. In the Fourier argument in Section~\ref{sec:fourier}, we choose the second factor to be an explicit harmonic exponential. This makes the Fourier variable explicit. There is only one unknown corrector in the integral identity, and no product of two such correctors. The integral remainder can be estimated by pairing $H$ and the differentiated corrector in $L^2$. Since the corrector estimate is linear in $\|H\|_{L^2}$, the resulting Fourier estimate is quadratic in this norm.
	\end{remark}
	
	\para{Compactly supported correctors} We now assume that $A=I+H$, $\supp H\subset K\Subset\Omega$, and $\Lambda_A=\Lambda_I$. Let $\zeta\in\C^n$ be a complex null vector. By the definition above, $e^{\zeta\cdot x}$ is harmonic.
	
	\begin{lemma}[Exterior-normalized complex exponential solutions]\label{lem:compact-corrector}
		There is a compact set $K_1$ satisfying $K\Subset \operatorname{int}(K_1)\subset K_1\Subset\Omega$, depending only on $K$ and $\Omega$, with the following property. For every complex null vector $\zeta\in\C^n$, there is a unique $A$-harmonic solution $u_\zeta\in H^1(\Omega)$ with boundary trace $e^{\zeta\cdot x}$, and
		\begin{equation}\label{eq:w-compact-support}
			w_\zeta:=u_\zeta-e^{\zeta\cdot x}=0
			\quad\text{in }\Omega\setminus K_1.
		\end{equation}
		Moreover,
		\begin{equation}\label{eq:r-def}
			u_\zeta=e^{\zeta\cdot x}(1+r_\zeta), \quad r_\zeta=e^{-\zeta\cdot x}w_\zeta, \quad
			\supp r_\zeta\subset K_1.
		\end{equation}
		In particular, $r_\zeta\in H^1(\Omega)$.
	\end{lemma}
	
	\begin{proof}
		Existence and uniqueness follow from uniform ellipticity. Set $f_\zeta=e^{\zeta\cdot x}|_{\p\Omega}$. Since $u_\zeta$ and $e^{\zeta\cdot x}$ have the same boundary trace, their difference $w_\zeta$ belongs to $H_0^1(\Omega)$. The exponential is harmonic because $\zeta\cdot\zeta=0$, and equality of the DN maps gives $(A\nabla u_\zeta)\cdot\nu=\p_\nu e^{\zeta\cdot x}$ on $\p\Omega$ in $H^{-1/2}(\p\Omega)$.
		
		Let $\nu$ be the outward unit normal on $\p\Omega$. Since $\p\Omega$ is smooth and compact, we may choose $\rho>0$, depending only on $K$ and $\Omega$, such that $2\rho<\operatorname{dist}(K,\p\Omega)$ and the map $\Phi(p,t)=p-t\nu(p)$ is a smooth diffeomorphism from $\p\Omega\times(-\rho,\rho)$ onto $\mathcal U=\{x\in\R^n:\operatorname{dist}(x,\p\Omega)<\rho\}$. The positive parameter range $0<t<\rho$ parametrizes $\mathcal U\cap\Omega$, while $-\rho<t<0$ parametrizes $\mathcal U\setminus\ol\Omega$. Thus, $\mathcal U$ is a two-sided collar of the boundary whose interior part is disjoint from $K$. In particular, $A=I$ and $\Delta w_\zeta=0$ in $\mathcal U\cap\Omega$.
		
		Extend $w_\zeta$ by zero outside $\Omega$, and denote the restriction of this extension to $\mathcal U$ by $\widetilde w_\zeta$. Since $w_\zeta\in H_0^1(\Omega)$, its zero extension belongs to $H^1(\R^n)$. For $\varphi\in C_c^\infty(\mathcal U)$, extend $\varphi$ by zero to $\R^n$ and use the same notation. Since $A=I$ wherever $\nabla\varphi$ can be nonzero in $\Omega$, the weak definitions of the DN maps give
		\begin{equation}\label{eq:zero-extension-harmonic}
			\begin{aligned}
				\int_{\mathcal U}\nabla\widetilde w_\zeta\cdot\nabla\varphi\,dx=\int_\Omega\bigl(A\nabla u_\zeta-\nabla e^{\zeta\cdot x}\bigr)\cdot\nabla\varphi\,dx=\langle(\Lambda_A-\Lambda_I)f_\zeta,\varphi|_{\p\Omega}\rangle=0.
			\end{aligned}
		\end{equation}
		Thus, $\widetilde w_\zeta$ is distributionally harmonic in the entire collar $\mathcal U$, including across $\p\Omega$.
		
		Weyl's lemma shows that $\widetilde w_\zeta$ is smooth and harmonic in $\mathcal U$. Every connected component of $\mathcal U$ contains a nonempty exterior open part, parametrized by $-\rho<t<0$, on which $\widetilde w_\zeta=0$. Unique continuation for harmonic functions implies $\widetilde w_\zeta=0$ on every component of $\mathcal U$. In particular, $w_\zeta=0$ at every point of $\Omega$ whose distance from $\p\Omega$ is less than $\rho$.
		
		Define $K_1=\{x\in\ol\Omega:\operatorname{dist}(x,\p\Omega)\ge\rho/2\}$. This is a compact subset of $\Omega$, depends only on $K$ and $\Omega$, and satisfies $K\Subset\operatorname{int}(K_1)$ by the choice of $\rho$. Since $\Omega\setminus K_1\subset\mathcal U\cap\Omega$, the preceding vanishing proves \eqref{eq:w-compact-support}. Finally, multiplication by the smooth nonvanishing function $e^{-\zeta\cdot x}$ preserves $H^1(\Omega)$ and does not change the support. Hence, $r_\zeta=e^{-\zeta\cdot x}w_\zeta\in H^1(\Omega)$, is supported in $K_1$, and satisfies \eqref{eq:r-def}.
	\end{proof}
	
	The function $u_\zeta$ has the usual complex geometrical optics (CGO) form, but no smallness of $r_\zeta$ is asserted at this stage. The point of Lemma~\ref{lem:compact-corrector} is that the correction has support in a fixed compact set. This support property uses both equality of the DN maps and the fact that $A=I$ near the boundary; it is not a property of the Dirichlet solution for an arbitrary conductivity. Therefore, the Carleman estimate can be applied directly to the actual correction, without multiplying it by a cutoff and introducing additional commutator terms. Its quantitative estimate is proved in Proposition~\ref{prop:corrector-estimate}. After extending $A$ by $I$ and $r_\zeta$ by zero outside $\Omega$, the same formula defines a whole space solution, which agrees with $e^{\zeta\cdot x}$ outside $K_1$.
	
	For a complex null vector $\zeta\in\C^n$, define $D_\zeta r=\nabla r+\zeta r$. Since $w_\zeta=e^{\zeta\cdot x}r_\zeta$, direct differentiation gives
	\begin{equation}\label{eq:D-zeta-relation}
		\nabla w_\zeta=e^{\zeta\cdot x}D_\zeta r_\zeta.
	\end{equation}
	
	\para{The conjugated equation} The divergence free condition gives the following conjugated equation.
	
	\begin{lemma}[Conjugated divergence-form equation]\label{lem:conjugated-equation}
		Suppose that $A=I+H$ and $\diver H=0$. Then the corrector in Lemma~\ref{lem:compact-corrector} satisfies
		\begin{equation}\label{eq:conjugated-equation}
			P_\zeta r_\zeta+Q_{H,\zeta}r_\zeta
			=\zeta^TH\zeta,
		\end{equation}
		where $P_\zeta=-\Delta-2\zeta\cdot\nabla$ and $\quad Q_{H,\zeta}r=-\diver(H\nabla r)-2H\zeta\cdot\nabla r-(\zeta^TH\zeta)r$. 
	\end{lemma}
	
	\begin{proof}
		For a smooth scalar function $r$,
		\begin{equation}\label{eq:conjugation-expansion}
			e^{-\zeta\cdot x}[-\diver(A\nabla)]e^{\zeta\cdot x}r
			=-\diver(A\nabla r)-2A\zeta\cdot\nabla r-(\zeta^TA\zeta)r,
		\end{equation}
		where we used $\diver(A\zeta)=0$ from $\diver A=0$. Since $A\in L^\infty(\Omega)$, both sides of \eqref{eq:conjugation-expansion} depend continuously on $r\in H^1(\Omega)$ as distributions. Approximation by smooth functions extends the identity to $H^1(\Omega)$. Apply \eqref{eq:conjugation-expansion} to $1+r_\zeta$. Since $\zeta\cdot\zeta=0$ and $A=I+H$, the equation for $u_\zeta=e^{\zeta\cdot x}(1+r_\zeta)$ becomes $-\Delta r_\zeta-2\zeta\cdot\nabla r_\zeta
			-\diver(H\nabla r_\zeta)-2H\zeta\cdot\nabla r_\zeta-(\zeta^TH\zeta)r_\zeta=\zeta^TH\zeta$, which is \eqref{eq:conjugated-equation}.
	\end{proof}

	\section{Carleman estimates for compactly supported functions and corrector bounds}\label{sec:carleman}
	
	Throughout Sections~\ref{sec:carleman} and \ref{sec:fourier}, we assume all the hypotheses of Theorem~\ref{thm:conductivity-main} except for the smallness condition on $\|H\|_{H^{s_*}(\R^n)}$, unless stated otherwise. We use the correctors constructed in Lemma~\ref{lem:compact-corrector}.
	
	\subsection{Semiclassical Sobolev spaces and the free estimate}
	
	For $h>0$ and $s\in\R$, define
	\begin{equation}\label{eq:semiclassical-Sobolev}
		\|v\|_{H^s_{\mathrm{scl}}(\R^n)}=\big\|\langle h\xi\rangle^s\wh v(\xi)\big\|_{L^2_\xi},
		\quad \langle h\xi\rangle=(1+h^2|\xi|^2)^{1/2},
	\end{equation}
	where the subscript $\xi$ indicates that the $L^2$ norm is taken in the Fourier variable. In particular, Plancherel's theorem gives
	\begin{equation}\label{eq:H1-scl-equivalence}
		\|v\|_{\Hs}^2\asymp\|v\|_{L^2}^2+\|h\nabla v\|_{L^2}^2.
	\end{equation}
	Here and below, $A\asymp B$ means that $cB\le A\le CB$ for constants $c,C>0$ independent of the parameters under consideration. In particular, in \eqref{eq:H1-scl-equivalence}, these constants are independent of $h$ and $v$. The $L^2$ pairing identifies $\Hsm$ with the continuous dual of $\Hs$, with norm equivalence constants independent of $h$. Let
	\begin{equation}\label{eq:null-frame-set}
		\mathcal Z=\left\{a+\mathsf{i}b:\, a,b\in\R^n,\ a\cdot b=0,\ |a|=|b|=1\right\}.
	\end{equation}
	This is the set of complex null vectors whose real and imaginary parts both have length one. For $\zeta_0\in\mathcal Z$, set
	$P_{\zeta_0,h}=-h^2\Delta-2h\zeta_0\cdot\nabla$.
	
	\begin{proposition}[Limiting Carleman estimate for compactly supported functions]\label{prop:compact-Carleman}
		Let $\mathcal O\subset\R^n$ be a bounded open set. There exist $h_0>0$ and $C>0$, depending only on $n$ and $\mathcal O$, such that
		\begin{equation}\label{eq:compact-Carleman}
			h\|v\|_{\Hs}\le C\|P_{\zeta_0,h}v\|_{\Hsm}
		\end{equation}
		for every $0<h<h_0$, every $\zeta_0\in\mathcal Z$, and every $v\in C_c^\infty(\mathcal O)$. The constants are independent of $h$, $\zeta_0$, and $v$. The same estimate \eqref{eq:compact-Carleman} holds for every $v\in H^1(\R^n)$ satisfying $\supp v\Subset\mathcal O$, where $P_{\zeta_0,h}v$ is understood in the distributional sense.
	\end{proposition}
	
	\begin{proof}
		Choose concentric Euclidean balls $B\Subset\widetilde B$, centered at the origin, such that $\ol{\mathcal O}\subset B$. Fix $a\in\mathbb S^{n-1}$ and set $\psi_a(x)=-a\cdot x$. Since $\nabla\psi_a=-a$ is constant and nonzero, $\psi_a$ is a limiting Carleman weight for the Euclidean metric; see \cite[Section~3]{DosSantosFerreiraKenigSaloUhlmann2009} and \cite[Section~2]{KrupchykUhlmann2014}. Apply \cite[Proposition~2.2]{KrupchykUhlmann2014} on $B$, using $\widetilde B$ as the larger ambient domain and setting the magnetic and electric potentials in that proposition equal to zero. In the notation of the cited result, $L_{0,0}=-\Delta$. Taking the Carleman weight to be $\psi_a$ gives constants $h_a>0$ and $C_a>0$ such that
		\begin{equation}\label{eq:real-Carleman}
			h\|v\|_{\Hs}\le C_a\big\|e^{-a\cdot x/h}(-h^2\Delta)e^{a\cdot x/h}v\big\|_{\Hsm}
		\end{equation}
		for every $v\in C_c^\infty(B)$ and every $0<h<h_a$. Indeed, the conjugated operator in \cite[Proposition~2.2]{KrupchykUhlmann2014} is $e^{\psi_a/h}(h^2L_{0,0})e^{-\psi_a/h}$, which is exactly $e^{-a\cdot x/h}(-h^2\Delta)e^{a\cdot x/h}$.
		
		The Sobolev orders in \eqref{eq:real-Carleman} follow from the estimate with a gain of two derivatives in \cite[Proposition~2.1 and the proof of Proposition~2.2]{KrupchykUhlmann2014}. Let $\psi_{a,\varepsilon}=\psi_a+\frac{h}{2\varepsilon}\psi_a^2$ be the convexified weight. Taking $s=-1$ in that estimate places the conjugated operator in $H^{-1}_{\mathrm{scl}}$ and controls the function in $H^1_{\mathrm{scl}}$, with the factor $h/\sqrt{\varepsilon}$ on the left. The parameter $\varepsilon>0$ is chosen sufficiently small and then fixed independently of $h$, so this factor is a fixed multiple of $h$. To pass back to the original weight, choose $\chi\in C_c^\infty(\widetilde B)$ such that $\chi=1$ in a neighborhood of $\ol B$. Applying the convexified estimate to $e^{\psi_a^2/(2\varepsilon)}v$ uses the identity $e^{-\psi_{a,\varepsilon}/h}e^{\psi_a^2/(2\varepsilon)}v=e^{-\psi_a/h}v$. Since $v$ is supported in $B$, multiplication by $e^{\psi_a^2/(2\varepsilon)}$ and its inverse may be replaced by multiplication by the compactly supported smooth functions $\chi e^{\psi_a^2/(2\varepsilon)}$ and $\chi e^{-\psi_a^2/(2\varepsilon)}$. These multiplication operators are bounded on $H^s_{\mathrm{scl}}(\R^n)$ for $-1\le s\le1$, with bounds independent of $h$ and $a\in\mathbb S^{n-1}$. This explains the $H^1_{\mathrm{scl}}$ and $H^{-1}_{\mathrm{scl}}$ norms in \eqref{eq:real-Carleman}. The choice $s=-1$ will be used below because the divergence-form perturbation naturally maps $H^1_{\mathrm{scl}}$ into $H^{-1}_{\mathrm{scl}}$ when only the $L^\infty$ norm of its coefficient is available.
		
		It remains to show that the constants in \eqref{eq:real-Carleman} may be chosen independently of $a\in\mathbb S^{n-1}$. Let $e_1=(1,0,\ldots,0)\in\R^n$, and let $C=C_{e_1}$ and $h_0=h_{e_1}$ be the constants supplied by \eqref{eq:real-Carleman} for the single direction $e_1$. For each $a\in\mathbb S^{n-1}$, choose an orthogonal map $R_a\in O(n)$ such that $R_ae_1=a$. Since $R_a^TR_a=I$, one has $R_a^Ta=e_1$, and $\psi_a(R_ax)=-a\cdot R_ax=-(R_a^Ta)\cdot x=-e_1\cdot x=\psi_{e_1}(x)$. Define $T_av=v\circ R_a$. Since $B$ and $\widetilde B$ are balls centered at the origin, they are invariant under $R_a$, and $T_a$ maps $C_c^\infty(B)$ onto itself. Orthogonality gives $\Delta(T_a\phi)=T_a(\Delta\phi)$ for every smooth function $\phi$. Moreover, $\wh{T_av}(\xi)=\wh v(R_a\xi)$, so the change of variables $\eta=R_a\xi$ and the identity $|R_a\xi|=|\xi|$ give $\|T_av\|_{H^s_{\mathrm{scl}}}=\|v\|_{H^s_{\mathrm{scl}}}$ for every $s\in\R$.
		
		The conjugated operators are intertwined by $T_a$. Indeed, $a\cdot R_ax=e_1\cdot x$, and hence $e^{\pm e_1\cdot x/h}T_av=T_a(e^{\pm a\cdot x/h}v)$. Combining this identity with $\Delta T_a=T_a\Delta$, we obtain
		\begin{equation*}
			\begin{aligned}
				e^{-e_1\cdot x/h}(-h^2\Delta)e^{e_1\cdot x/h}T_av=e^{-e_1\cdot x/h}(-h^2\Delta)T_a\bigl(e^{a\cdot x/h}v\bigr)=T_a\bigl(e^{-a\cdot x/h}(-h^2\Delta)e^{a\cdot x/h}v\bigr).
			\end{aligned}
		\end{equation*}
		For $0<h<h_0$, apply the special case $a=e_1$ of \eqref{eq:real-Carleman} to $T_av$. The preceding identities give
		\begin{equation*}
			\begin{aligned}
				h\|v\|_{\Hs}=h\|T_av\|_{\Hs}\le C\big\|e^{-e_1\cdot x/h}(-h^2\Delta)e^{e_1\cdot x/h}T_av\big\|_{\Hsm}=C\big\|e^{-a\cdot x/h}(-h^2\Delta)e^{a\cdot x/h}v\big\|_{\Hsm}.
			\end{aligned}
		\end{equation*}
		This shows that the same constants $C=C_{e_1}$ and $h_0=h_{e_1}$ may be used in \eqref{eq:real-Carleman} for every $a\in\mathbb S^{n-1}$. After decreasing $h_0$, we may assume that $0<h_0\le1$.
		
		Let $b\in\mathbb S^{n-1}$ satisfy $a\cdot b=0$, and define $M_bv=e^{\mathsf{i}b\cdot x/h}v$. With the Fourier transform convention in Section~\ref{sec:fourier}, one has $\wh{M_bv}(\xi)=\wh v(\xi-b/h)$. After the change of variables $\eta=\xi-b/h$, one obtains
		\[
		\|M_bv\|_{H^s_{\mathrm{scl}}}^2=\int_{\R^n}\langle h\eta+b\rangle^{2s}|\wh v(\eta)|^2\,d\eta.
		\] 
		The triangle inequality, applied once with $b$ and once with $-b$, shows that $\langle h\eta+b\rangle\asymp\langle h\eta\rangle$ uniformly in $h>0$, $\eta\in \R^n$, and $b\in\mathbb S^{n-1}$. For $-1\le s\le1$, this gives
		\begin{equation}\label{eq:modulation-Sobolev}
			C_s^{-1}\|v\|_{H^s_{\mathrm{scl}}}\le\|M_bv\|_{H^s_{\mathrm{scl}}}\le C_s\|v\|_{H^s_{\mathrm{scl}}},
		\end{equation}
		and the same inequalities hold with $M_b$ replaced by $M_b^{-1}$.
		
		Set $L_{a,h}\cdot=e^{-a\cdot x/h}(-h^2\Delta)(e^{a\cdot x/h}\cdot)$. Direct differentiation gives $L_{a,h}=-h^2\Delta-2ha\cdot\nabla-|a|^2$, while $M_b^{-1}\nabla (M_b\cdot)=(\nabla+\mathsf{i}b/h)\cdot$. Combining these identities gives
		\begin{equation}\label{eq:modulated-Carleman-operator}
			M_b^{-1}L_{a,h}M_b=-h^2\Delta-2h(a+\mathsf{i}b)\cdot\nabla+\bigl(|b|^2-|a|^2-2\mathsf{i}a\cdot b\bigr).
		\end{equation}
		The zeroth-order term vanishes because $|a|=|b|=1$ and $a\cdot b=0$. For $\zeta_0=a+\mathsf{i}b$, we have
		\begin{equation}\label{eq:modulated-operator-identity}
			M_b^{-1}L_{a,h}M_b=P_{\zeta_0,h}\quad\text{equivalently}\quad L_{a,h}M_b=M_bP_{\zeta_0,h}.
		\end{equation}
		
		Let $v\in C_c^\infty(\mathcal O)$. Since multiplication by $M_b$ does not change the support, one has $M_bv\in C_c^\infty(\mathcal O)\subset C_c^\infty(B)$. Applying the estimate \eqref{eq:real-Carleman} to $M_bv$ and using the modulation estimates \eqref{eq:modulation-Sobolev} and the operator identity \eqref{eq:modulated-operator-identity}, we obtain
		\begin{equation*}
			\begin{aligned}
				h\|v\|_{\Hs}\le Ch\|M_bv\|_{\Hs}\le C\|L_{a,h}M_bv\|_{\Hsm}=C\|M_bP_{\zeta_0,h}v\|_{\Hsm}\le C\|P_{\zeta_0,h}v\|_{\Hsm}.
			\end{aligned}
		\end{equation*}
		This proves \eqref{eq:compact-Carleman} for $v\in C_c^\infty(\mathcal O)$, with constants independent of $\zeta_0\in\mathcal Z$.
		
		It remains to extend the estimate to compactly supported $H^1$ functions. Let $v\in H^1(\R^n)$ satisfy $\supp v\Subset\mathcal O$, and choose an open set $\mathcal O_1$ such that $\supp v\Subset\mathcal O_1\Subset\mathcal O$. Standard mollification gives a sequence $v_k\in C_c^\infty(\mathcal O_1)\subset C_c^\infty(\mathcal O)$ such that $v_k\to v$ in $H^1(\R^n)$. The Fourier multiplier of $P_{\zeta_0,h}$ is $p_{\zeta_0,h}(\xi)=h^2|\xi|^2-2\mathsf{i}h\zeta_0\cdot\xi$. Since $|\zeta_0|=\sqrt2$ on $\mathcal Z$, one has $|p_{\zeta_0,h}(\xi)|\le C\langle h\xi\rangle^2$ uniformly for $0<h\le1$ and $\zeta_0\in\mathcal Z$. Thus,
		\begin{equation}\label{eq:P-H1-Hminus1}
			\begin{aligned}
				\|P_{\zeta_0,h}f\|_{\Hsm}^2=\int_{\R^n}\langle h\xi\rangle^{-2}|p_{\zeta_0,h}(\xi)|^2|\wh f(\xi)|^2\,d\xi\le C\int_{\R^n}\langle h\xi\rangle^2|\wh f(\xi)|^2\,d\xi
				=C\|f\|_{\Hs}^2.
			\end{aligned}
		\end{equation}
		Thus, the operator $P_{\zeta_0,h}:H^1_{\mathrm{scl}}(\R^n)\to H^{-1}_{\mathrm{scl}}(\R^n)$ is continuous, uniformly for the stated values of $h$ and $\zeta_0$. The convergence $v_k\to v$ in $H^1(\R^n)$ implies convergence in $H^1_{\mathrm{scl}}(\R^n)$ for each fixed $0<h<h_0$, and \eqref{eq:P-H1-Hminus1} gives $P_{\zeta_0,h}v_k\to P_{\zeta_0,h}v$ in $H^{-1}_{\mathrm{scl}}(\R^n)$. Passing to the limit in \eqref{eq:compact-Carleman} proves the final assertion.
	\end{proof}
	
	\para{The divergence-form perturbation} For $\zeta_0\in\mathcal Z$, define
	\begin{equation}\label{eq:Q-h-def}
		Q_{H,\zeta_0,h}\cdot
		=-h^2\diver(H\nabla \cdot)-2hH\zeta_0\cdot\nabla \cdot-(\zeta_0^TH\zeta_0)\cdot.
	\end{equation}
	
	\begin{lemma}[Perturbation estimate]\label{lem:perturbation-estimate}
		There is a constant $C$, depending only on the dimension, such that
		\begin{equation}\label{eq:perturbation-estimate}
			\|Q_{H,\zeta_0,h}v\|_{\Hsm}
			\le C\|H\|_{L^\infty}\|v\|_{\Hs}
		\end{equation}
		for all $0<h\le1$, $\zeta_0\in\mathcal Z$, and $v\in\Hs$. Moreover,
		\begin{equation}\label{eq:source-Hminus1}
			\|\zeta_0^TH\zeta_0\|_{\Hsm}
			\le C\|H\|_{L^2}.
		\end{equation}
	\end{lemma}
	
	\begin{proof}
		Using the distributional definition of the divergence term, for every $\psi\in\Hs$ we have
		\begin{equation}\label{eq:perturbation-duality}
			\begin{aligned}
				|\langle Q_{H,\zeta_0,h}v,\psi\rangle|\le h^2\|H\|_{L^\infty}\|\nabla v\|_{L^2}\|\nabla\psi\|_{L^2}
				+2h\|H\|_{L^\infty}|\zeta_0|\|\nabla v\|_{L^2}\|\psi\|_{L^2}+\|H\|_{L^\infty}|\zeta_0|^2\|v\|_{L^2}\|\psi\|_{L^2}.
			\end{aligned}
		\end{equation}
		By the definition of the semiclassical norms and Plancherel's theorem, the constants in \eqref{eq:H1-scl-equivalence} depend only on $n$. Since $|\zeta_0|=\sqrt2$ for $\zeta_0\in\mathcal Z$, the right-hand side of \eqref{eq:perturbation-duality} is bounded by $C\|H\|_{L^\infty}\allowbreak\|v\|_{\Hs}\allowbreak\|\psi\|_{\Hs}$, where $C$ depends only on $n$. Taking the supremum over $\psi$ with $\|\psi\|_{\Hs}=1$ proves \eqref{eq:perturbation-estimate}.
		
		For the source term, the Cauchy--Schwarz inequality gives $|\zeta_0^TH(x)\zeta_0|\le|\zeta_0|^2|H(x)|=2|H(x)|$. Also, by duality and \eqref{eq:H1-scl-equivalence}, $\|F\|_{\Hsm}\le C\|F\|_{L^2}$ for every $F\in L^2(\R^n)$, with $C$ depending only on $n$. Hence, $\|\zeta_0^TH\zeta_0\|_{\Hsm}\le 2C\|H\|_{L^2}$, which proves \eqref{eq:source-Hminus1}. 
	\end{proof}
	
	\subsection{Corrector estimates in a growing complex frequency range} Let $K_1$ be the compact set fixed in Lemma~\ref{lem:compact-corrector}. Thus, $K\Subset\operatorname{int}(K_1)\subset K_1\Subset\Omega$, and both $w_\zeta$ and $r_\zeta$ are supported in $K_1$ for every complex null vector $\zeta$. We regard $r_\zeta$ as its zero extension to $\R^n$. Put
	\begin{equation}\label{eq:X-delta}
		X=\|H\|_{L^2(\R^n)}, \quad \delta=\|H\|_{L^\infty}.
	\end{equation}
	All constants below may depend on $n$, $\Omega$, and the fixed set $K$, but are independent of $H$, $X$, $\delta$, and the complex frequency $\zeta$, unless stated otherwise. This includes the dependence on $K_1$, since $K_1$ was fixed in Lemma~\ref{lem:compact-corrector} depending only on $K$ and $\Omega$.
	
	\begin{proposition}[Corrector estimate]\label{prop:corrector-estimate}
		There are constants $c_0,C,\delta_0>0$, depending only on $n$, $\Omega$, and the fixed set $K$, with the following properties. If $\delta=0$, then $H=0$ almost everywhere, $A=I$, and $r_\zeta=0$ for every complex null vector $\zeta$. Suppose that $0<\delta\le\delta_0$. For every complex null vector $\zeta\in\C^n$ satisfying
		\begin{equation}\label{eq:corrector-frequency-range}
			|\zeta|\le c_0\delta^{-1},
		\end{equation}
		one has
		\begin{equation}\label{eq:corrector-bound-final}
			\|D_\zeta r_\zeta\|_{L^2(\R^n)}\le C|\zeta|(1+|\zeta|)X.
		\end{equation}
		More precisely,
		\begin{equation}\label{eq:corrector-small-frequency}
			\|D_\zeta r_\zeta\|_{L^2(\R^n)}\le C|\zeta|X \quad\text{if }|\zeta|\le1,
		\end{equation}
		and
		\begin{equation}\label{eq:corrector-large-frequency}
			\|D_\zeta r_\zeta\|_{L^2(\R^n)}\le C|\zeta|^2X \quad\text{if }1\le|\zeta|\le c_0\delta^{-1}.
		\end{equation}
	\end{proposition}
	
	\begin{proof}
		If $\delta=0$, then $H=0$ almost everywhere and $A=I$. The functions $u_\zeta$ and $e^{\zeta\cdot x}$ solve the same Dirichlet problem with $u_\zeta|_{\p\Omega}=e^{\zeta\cdot x}|_{\p\Omega}$. Uniqueness gives $u_\zeta=e^{\zeta\cdot x}$ in $\overline{\Omega}$, so $w_\zeta=0$ and $r_\zeta=0$.
		
		Assume from now on that $0<\delta\le\delta_0$. Fix a bounded open set $\mathcal O\subset\R^n$ such that $K_1\Subset\mathcal O$, and let $h_0$ be the constant in Proposition~\ref{prop:compact-Carleman} for this fixed set $\mathcal O$. Choose once and for all $R_0>\max\big\{1,\frac{\sqrt2}{h_0}\big\}$. We first prove the estimate for $|\zeta|\le R_0$. Recall that $w_\zeta=u_\zeta-e^{\zeta\cdot x}\in H_0^1(\Omega)$ and
		\begin{equation}\label{eq:w-equation}
			-\diver(A\nabla w_\zeta)=\diver\big(e^{\zeta\cdot x}H\zeta\big) \quad\text{in }\Omega.
		\end{equation}
		We shall choose $\delta_0\le1/2$. Since $H$ is real symmetric and $\|H\|_{L^\infty}=\delta$, every eigenvalue of $H(x)$ belongs to $[-\delta,\delta]$ for almost every $x$. Every eigenvalue of $A(x)=I+H(x)$ is at least $1-\delta\ge1/2$. In particular, $A(x)z\cdot\overline z\ge\frac12|z|^2$ for every $z\in\C^n$ and almost every $x\in\Omega$.
		
		Testing \eqref{eq:w-equation} with $\overline{w_\zeta}$ and taking real parts gives
		\begin{equation}\label{eq:w-energy}
			\begin{aligned}
				\frac12\|\nabla w_\zeta\|_{L^2(\Omega)}^2\le\operatorname{Re}\int_\Omega A\nabla w_\zeta\cdot\nabla\overline{w_\zeta}\,dx =-\operatorname{Re}\int_\Omega e^{\zeta\cdot x}H\zeta\cdot\nabla\overline{w_\zeta}\,dx\le|\zeta|\|e^{\zeta\cdot x}H\|_{L^2(K)}\|\nabla w_\zeta\|_{L^2(\Omega)}.
			\end{aligned}
		\end{equation}
		If $\|\nabla w_\zeta\|_{L^2(\Omega)}=0$, then $w_\zeta=0$ because $w_\zeta\in H_0^1(\Omega)$, and \eqref{eq:D-zeta-relation} gives $D_\zeta r_\zeta=0$. Otherwise, division by $\|\nabla w_\zeta\|_{L^2(\Omega)}$ gives
		\begin{equation}\label{eq:w-energy-final}
			\|\nabla w_\zeta\|_{L^2(\Omega)}\le C|\zeta|\|e^{\zeta\cdot x}H\|_{L^2(K)}.
		\end{equation}
		
		Since the fixed set $K_1$ is compact, choose $\rho>0$ such that $K_1\subset\{x\in\R^n:\, |x|\le\rho\}$. The inclusion $K\subset K_1$ and the inequality $|e^{\pm\zeta\cdot x}|\le e^{|\operatorname{Re}\zeta||x|}$ give $|e^{\pm\zeta\cdot x}|\le e^{\rho|\operatorname{Re}\zeta|}$ on $K_1$. In particular, $\|e^{\zeta\cdot x}H\|_{L^2(K)}\le e^{\rho|\operatorname{Re}\zeta|}X$. Lemma~\ref{lem:compact-corrector} gives $\supp w_\zeta\subset K_1$, while \eqref{eq:D-zeta-relation} gives $D_\zeta r_\zeta=e^{-\zeta\cdot x}\nabla w_\zeta$. Using \eqref{eq:w-energy-final}, we obtain
		\begin{equation}\label{eq:bounded-frequency-corrector}
			\begin{aligned}
				\|D_\zeta r_\zeta\|_{L^2(\R^n)}=\|e^{-\zeta\cdot x}\nabla w_\zeta\|_{L^2(K_1)}\le e^{\rho|\operatorname{Re}\zeta|}\|\nabla w_\zeta\|_{L^2(\Omega)}\le C|\zeta|e^{2\rho|\operatorname{Re}\zeta|}X.
			\end{aligned}
		\end{equation}
		For $|\zeta|\le R_0$, the inequality $|\operatorname{Re}\zeta|\le|\zeta|$ gives
		\begin{equation}\label{eq:fixed-frequency-corrector}
			\|D_\zeta r_\zeta\|_{L^2(\R^n)}\le C_{R_0}|\zeta|X,
		\end{equation}
		where $C_{R_0}=Ce^{2\rho R_0}$. This proves \eqref{eq:corrector-small-frequency}. If $1\le|\zeta|\le R_0$, then $|\zeta|\le|\zeta|^2$, and \eqref{eq:fixed-frequency-corrector} also gives
		\begin{equation}\label{eq:bounded-transition-corrector}
			\|D_\zeta r_\zeta\|_{L^2(\R^n)}\le C_{R_0}|\zeta|^2X.
		\end{equation}
		
		We next consider $|\zeta|\ge R_0$. Write $\zeta=\alpha+\mathsf{i}\beta$ with $\alpha,\beta\in\R^n$. Since $\zeta\cdot\zeta=0$, one has $\alpha\cdot\beta=0$ and $|\alpha|=|\beta|=:s$. Set $h=s^{-1}$ and $\zeta_0=h\zeta=\frac{\alpha}{s}+\mathsf{i}\frac{\beta}{s}$, then $\zeta_0\in\mathcal Z$, $|\zeta_0|=\sqrt2$, and
		\begin{equation}\label{eq:h-zeta-relation}
			h=\frac{\sqrt2}{|\zeta|}.
		\end{equation}
		The choice of $R_0$ gives $0<h<h_0$.
		
		Since $r_\zeta=0$ and $H=0$ in a neighborhood of $\p\Omega$, extending $r_\zeta$ by zero and $A$ by $I$ introduces no distribution supported on $\p\Omega$. Multiplying \eqref{eq:conjugated-equation} by $h^2$ and using $\zeta=h^{-1}\zeta_0$ gives
		\begin{equation}\label{eq:semiclassical-corrector-equation}
			P_{\zeta_0,h}r_\zeta+Q_{H,\zeta_0,h}r_\zeta=\zeta_0^TH\zeta_0
		\end{equation}
		in distributions on $\R^n$. The zero extension of $r_\zeta$ belongs to $H^1(\R^n)$ and is supported in the fixed set $K_1\Subset\mathcal O$. Proposition~\ref{prop:compact-Carleman} applies to this extension. Using \eqref{eq:semiclassical-corrector-equation} and Lemma~\ref{lem:perturbation-estimate}, we obtain
		\begin{equation}\label{eq:corrector-absorption}
			\begin{aligned}
				h\|r_\zeta\|_{\Hs}\le C\|P_{\zeta_0,h}r_\zeta\|_{\Hsm}\le C\|\zeta_0^TH\zeta_0\|_{\Hsm}
				+C\|Q_{H,\zeta_0,h}r_\zeta\|_{\Hsm}\le C_1X+C_2\delta\|r_\zeta\|_{\Hs},
			\end{aligned}
		\end{equation}
		where the constants $C_1$ and $C_2$ depend only on $n$, the fixed set $\mathcal O$, and the fixed support sets $K$ and $K_1$. In particular, they are independent of $H$, $X$, $\delta$, $\zeta$, and $h$. Since $\mathcal O$ and $K_1$ have already been fixed in terms of $K$ and $\Omega$, the constants $C_1$ and $C_2$ are fixed before $c_0$ and $\delta_0$ are chosen.
		
		Choose $c_0>0$ so that $C_2c_0\le1/\sqrt2$. If $|\zeta|\le c_0\delta^{-1}$, then $\delta\le c_0|\zeta|^{-1}$. Formula~\eqref{eq:h-zeta-relation} gives 
		\begin{equation}\label{eq:estimate-constant-C2}
			C_2\delta\le\frac{C_2c_0}{|\zeta|}\le\frac{1}{\sqrt2|\zeta|}=\frac h2.
		\end{equation}
		By \eqref{eq:estimate-constant-C2}, the last term in \eqref{eq:corrector-absorption} is bounded by $\frac h2\|r_\zeta\|_{\Hs}$. Moving this term to the left and dividing by $h/2$ gives
		\begin{equation}\label{eq:r-H1-bound}
			\|r_\zeta\|_{\Hs}\le Ch^{-1}X.
		\end{equation}
		Since $\zeta=h^{-1}\zeta_0$, one has $D_\zeta r_\zeta=h^{-1}(h\nabla+\zeta_0)r_\zeta$. Formula~\eqref{eq:H1-scl-equivalence}, the identity $|\zeta_0|=\sqrt2$, and \eqref{eq:r-H1-bound} give
		\begin{equation}\label{eq:Dzeta-from-H1}
			\begin{aligned}
				\|D_\zeta r_\zeta\|_{L^2(\R^n)}\le h^{-1}\big(\|h\nabla r_\zeta\|_{L^2}+|\zeta_0|\|r_\zeta\|_{L^2}\big)	\le Ch^{-1}\|r_\zeta\|_{\Hs}\le Ch^{-2}X\le C|\zeta|^2X.
			\end{aligned}
		\end{equation}
		This proves \eqref{eq:corrector-large-frequency} for $R_0\le|\zeta|\le c_0\delta^{-1}$. Formula~\eqref{eq:bounded-transition-corrector} gives the same estimate for $1\le|\zeta|\le R_0$.
		
		Choose $0<\delta_0\le\min\big\{\frac12,\frac{c_0}{R_0}\big\}$. For every $0<\delta\le\delta_0$, one has $c_0\delta^{-1}\ge R_0$, so the bounded frequency estimate and the Carleman estimate cover the entire interval $1\le|\zeta|\le c_0\delta^{-1}$. After increasing $C$ to include the fixed constant $C_{R_0}$, this proves \eqref{eq:corrector-small-frequency} and \eqref{eq:corrector-large-frequency} with one constant $C$.
		
		For $|\zeta|\le1$, the right-hand side in \eqref{eq:corrector-small-frequency} is bounded by the right-hand side in \eqref{eq:corrector-bound-final}. For $1\le|\zeta|\le c_0\delta^{-1}$, one has $|\zeta|^2\le|\zeta|(1+|\zeta|)$. This proves \eqref{eq:corrector-bound-final}.
	\end{proof}
	
	\begin{remark}\label{rem:corrector-frequency-range}
		The upper bound in \eqref{eq:corrector-frequency-range} is the condition needed to absorb the perturbation term in \eqref{eq:corrector-absorption}. The relation $h=\sqrt2/|\zeta|$ turns $C_2\delta\le h/2$ into a bound of the form $|\zeta|\le c_0\delta^{-1}$. In the Fourier argument below, the complex frequency associated with a real frequency $\xi$ has size $|\zeta|=|\xi|/\sqrt2$. For $X>0$, the frequency decomposition uses $R=\kappa X^{-1/s_*}$, where $\kappa>0$ is fixed and sufficiently small. If $M=\|H\|_{H^{s_*}}$, Lemma~\ref{lem:GN} gives $\delta\le CX^{1/s_*}M^{1-1/s_*}$, and $R\delta\le C\kappa M^{1-1/s_*}$. Smallness of $M$ places this radius in the real-frequency range of Proposition~\ref{prop:Fourier-estimate}.
	\end{remark}

	\section{Fourier estimates and Sobolev rigidity}\label{sec:fourier}\label{sec:conductivity-proof}
	
	We now use the corrector estimate in Proposition~\ref{prop:corrector-estimate} to control the Fourier transform of $H$ and then combine this estimate with the Sobolev regularity of $H$. The quantity $\delta=\|H\|_{L^\infty}$ defined in \eqref{eq:X-delta} continues to use the pointwise matrix operator norm, as specified earlier. In particular, $|Hv|\le\delta|v|$ for every $v\in\C^n$ and almost every $x\in\R^n$.
	
	\subsection{Fourier estimates in a growing frequency range}
	
	We use the Fourier transform convention $\wh f(\xi)=\int_{\R^n}e^{-\mathsf{i}x\cdot\xi}f(x)\,dx$, and apply it componentwise to vector and matrix fields. Since $H$ is compactly supported and belongs to $L^2(\R^n)$, it also belongs to $L^1(\R^n)$, so $\wh H$ is continuous.
	
	For each $1\le j\le n$, the identity $\p_iH^{ij}=0$ gives $\mathsf{i}\sum_{i=1}^n\xi_i\wh H^{ij}(\xi)=0$. In matrix notation, this says $\xi^T\wh H(\xi)=0$. The tensor $H$ is symmetric, so $\wh H(\xi)^T=\wh H(\xi)$ and $\wh H(\xi)\xi=0$ for every $\xi\in\R^n$.
	
	\begin{lemma}[Null decomposition of a real frequency]\label{lem:null-decomposition}
		Let $0\ne\xi\in\R^n$, set $\omega=\xi/|\xi|$, and let $\theta\in\mathbb S^{n-1}$ satisfy $\theta\cdot\xi=0$. Define $\zeta=\frac{|\xi|}{2}(\theta+\mathsf{i}\omega)$ and $\eta=\frac{|\xi|}{2}(-\theta+\mathsf{i}\omega)$. Then $\zeta$ and $\eta$ are complex null vectors, $\zeta+\eta=\mathsf{i}\xi$, and $|\zeta|=|\eta|=|\xi|/\sqrt2$. If $M\in\C^{n\times n}$ is symmetric and $M\xi=0$, then $\zeta^TM\eta=-\frac{|\xi|^2}{4}\theta^TM\theta$.
	\end{lemma}
	
	\begin{proof}
		The orthogonality $\theta\cdot\omega=0$ and the identities $|\theta|=|\omega|=1$ give $\zeta\cdot\zeta=\eta\cdot\eta=0$, while the definitions give $\zeta+\eta=\mathsf{i}|\xi|\omega=\mathsf{i}\xi$. The Hermitian norms satisfy $|\zeta|^2=|\eta|^2=|\xi|^2/2$. Since $M\xi=0$, one has $M\omega=0$. Symmetry also gives $\omega^TM=0$. Expanding $\zeta^TM\eta$ leaves only the term containing $\theta^TM\theta$, and gives $\zeta^TM\eta=-\frac{|\xi|^2}{4}\theta^TM\theta$.
	\end{proof}
	
	\begin{lemma}[Polarization on the transverse space]\label{lem:matrix-polarization}
		For every $n\ge2$, there is $C_n>0$ such that, if $M\in\C^{n\times n}$ is symmetric and $M\xi=0$ for some $0\ne\xi\in\R^n$, then
		\begin{equation}\label{eq:matrix-polarization}
			|M|\le C_n\sup\left\{|\theta^TM\theta|:\,  \theta\in\R^n,\ \theta\perp\xi,\ |\theta|=1\right\}.
		\end{equation}
	\end{lemma}
	
	\begin{proof}
		Set $\omega=\xi/|\xi|$ and choose an orthonormal basis $e_1,\ldots,e_{n-1},e_n$ of $\R^n$ with $e_n=\omega$. Since $M\omega=0$, the last column of $M$ in this basis vanishes. Symmetry shows that the last row also vanishes. It is enough to estimate the entries $M_{jk}=e_j^TMe_k$ for $1\le j,k\le n-1$.
		
		Let $Q=\sup\{|\theta^TM\theta|:\, \theta\in \R^n, \ \theta\perp\xi,\ |\theta|=1\}$. For each $1\le j\le n-1$, one has $|M_{jj}|=|e_j^TMe_j|\le Q$. If $j\ne k$, set $\theta_{jk}=(e_j+e_k)/\sqrt2$. Symmetry gives $2M_{jk}=2\theta_{jk}^TM\theta_{jk}-e_j^TMe_j-e_k^TMe_k$, so $|M_{jk}|\le2Q$. All remaining entries vanish. It follows that $|M|^2=\sum_{j,k=1}^{n-1}|M_{jk}|^2\le C_n^2Q^2$, which proves \eqref{eq:matrix-polarization}.
	\end{proof}
	
	\begin{proposition}[Fourier estimate in a growing frequency range]\label{prop:Fourier-estimate}
		Let $c_0$ and $\delta_0$ be the constants in Proposition~\ref{prop:corrector-estimate}. There are constants $c_1,C>0$, depending only on $n$, $\Omega$, and the fixed set $K$, with the following property. If $0<\delta\le\delta_0$, then
		\begin{equation}\label{eq:Fourier-estimate}
			|\wh H(\xi)|\le C(1+|\xi|)X^2\quad\text{for every }|\xi|\le c_1\delta^{-1}.
		\end{equation}
		If $\delta=0$, then $H=0$ almost everywhere and $\wh H=0$.
	\end{proposition}
	
	\begin{proof}
		Assume that $0<\delta\le\delta_0$, and choose $c_1>0$ so that $c_1\le\sqrt2c_0$. Fix $0\ne\xi\in\R^n$ with $|\xi|\le c_1\delta^{-1}$, and let $\theta\in\mathbb S^{n-1}$ satisfy $\theta\cdot\xi=0$. Define $\zeta$ and $\eta$ as in Lemma~\ref{lem:null-decomposition}. Then $|\zeta|=|\xi|/\sqrt2\le c_0\delta^{-1}$, so Proposition~\ref{prop:corrector-estimate} applies to $r_\zeta$.
		
		Using Lemma~\ref{lem:compact-corrector}, the function $u_\zeta=e^{\zeta\cdot x}(1+r_\zeta)$ is $A$-harmonic, while $e^{\eta\cdot x}$ is harmonic. Lemma~\ref{lem:mixed-identity} gives $0=\int_\Omega H\nabla u_\zeta\cdot\nabla e^{\eta\cdot x}\,dx$. Since $\nabla u_\zeta=e^{\zeta\cdot x}(\zeta+D_\zeta r_\zeta)$ and $\nabla e^{\eta\cdot x}=e^{\eta\cdot x}\eta$, one has $H\nabla u_\zeta\cdot\nabla e^{\eta\cdot x}=e^{(\zeta+\eta)\cdot x}\big((H\zeta)\cdot\eta+(HD_\zeta r_\zeta)\cdot\eta\big)$. Since $\supp H\subset K$, the mixed identity gives
		\begin{equation}\label{eq:mixed-id-H}
			\int_K e^{(\zeta+\eta)\cdot x}(H\zeta)\cdot\eta\,dx=-\int_K e^{(\zeta+\eta)\cdot x}(HD_\zeta r_\zeta)\cdot\eta\,dx.
		\end{equation}
		Using $H^T=H$, $\zeta+\eta=\mathsf{i}\xi$, the Fourier transform convention, and \eqref{eq:mixed-id-H}, we obtain
		\begin{equation}\label{eq:Fourier-identity-expanded}
			\begin{aligned}
				\zeta^T\wh H(-\xi)\eta
				&=\zeta^T\bigg(\int_{\R^n}e^{\mathsf{i}\xi\cdot x}H(x)\,dx\bigg)\eta=\int_K e^{\mathsf{i}\xi\cdot x}\zeta^TH(x)\eta\,dx=\int_K e^{(\zeta+\eta)\cdot x}(H(x)\zeta)\cdot\eta\,dx\\
				&=-\int_K e^{(\zeta+\eta)\cdot x}(H(x)D_\zeta r_\zeta)\cdot\eta\,dx=-\int_K e^{\mathsf{i}\xi\cdot x}(H(x)D_\zeta r_\zeta)\cdot\eta\,dx.
			\end{aligned}
		\end{equation}	
		
		Since $\wh H(-\xi)\xi=0$ and $\wh H(-\xi)$ is symmetric, Lemma~\ref{lem:null-decomposition} gives $\zeta^T\wh H(-\xi)\eta=-\frac{|\xi|^2}{4}\theta^T\wh H(-\xi)\theta$. The Cauchy--Schwarz inequality gives
		\begin{equation}\label{eq:Fourier-remainder-L2}
			\left|\int_K e^{\mathsf{i}\xi\cdot x}(H D_\zeta r_\zeta)\cdot\eta\,dx\right|
			\le
			|\eta|\|H\|_{L^2(K)}\|D_\zeta r_\zeta\|_{L^2(K)}
			\le
			|\eta|X\|D_\zeta r_\zeta\|_{L^2(\R^n)}.
		\end{equation}
		Here the pointwise matrix norm in the $L^2$ norm is the Frobenius norm fixed above, and $|Hz|\le |H||z|$ for every $z\in\C^n$.
		
		Since $|\eta|=|\xi|/\sqrt2$, if $|\xi|\le2$, then \eqref{eq:corrector-bound-final} and $|\zeta|=|\xi|/\sqrt2$ give $\|D_\zeta r_\zeta\|_{L^2(\R^n)}\le C|\zeta|(1+|\zeta|)X\le C|\xi|X$. The right-hand side of \eqref{eq:Fourier-remainder-L2} is bounded by $C|\xi|^2X^2$. If $2\le|\xi|\le c_1\delta^{-1}$, then \eqref{eq:corrector-large-frequency} gives $\|D_\zeta r_\zeta\|_{L^2(\R^n)}\le C|\xi|^2X$, and the right-hand side of \eqref{eq:Fourier-remainder-L2} is bounded by $C|\xi|^3X^2$. Using \eqref{eq:Fourier-identity-expanded} and dividing by $|\xi|^2$ gives $|\theta^T\wh H(-\xi)\theta| \le C(1+|\xi|)X^2$. The vector $\theta$ was arbitrary among the unit vectors in $\xi^\perp$. Lemma~\ref{lem:matrix-polarization}, applied to $M=\wh H(-\xi)$, gives $|\wh H(-\xi)|\le C(1+|\xi|)X^2$. Replacing $\xi$ by $-\xi$ proves \eqref{eq:Fourier-estimate} for every nonzero frequency in the stated range. Continuity of $\wh H$ gives the estimate at $\xi=0$.
	\end{proof}
	
	\subsection{Frequency splitting and proof of the conductivity theorem}
	
	\para{Low and high Fourier modes}
	
	\begin{lemma}[Low-frequency estimate]\label{lem:low-frequency}
		Under the hypotheses of Proposition~\ref{prop:Fourier-estimate}, suppose that $0<\delta\le\delta_0$. Then
		\begin{equation}\label{eq:low-frequency-estimate}
			\|\wh H\|_{L^2(|\xi|\le R)} \le CR^{s_*}X^2
		\end{equation}
		for every $1\le R\le c_1\delta^{-1}$.
	\end{lemma}
	
	\begin{proof}
		Applying Proposition~\ref{prop:Fourier-estimate} and using polar coordinates, we obtain
		\begin{equation*}
			\int_{|\xi|\le R}|\wh H(\xi)|^2\,d\xi\le CX^4\int_0^R(1+r)^2r^{n-1}\,dr
			\le CX^4R^{n+2}.
		\end{equation*}
		The last inequality uses $R\ge1$. Taking square roots and using $s_*=(n+2)/2$ proves \eqref{eq:low-frequency-estimate}.
	\end{proof}
	
	\begin{lemma}[Sobolev tail]\label{lem:Sobolev-tail}
		Let $m>0$ and set $M_m=\|H\|_{H^m(\R^n)}$. For every $R\ge1$,
		\begin{equation}\label{eq:Sobolev-tail}
			\|\wh H\|_{L^2(|\xi|>R)}\le CR^{-m}M_m.
		\end{equation}
	\end{lemma}
	
	\begin{proof}
		On the set $|\xi|>R$, one has $(1+|\xi|^2)^m\ge R^{2m}$. It follows that
		\[
		\int_{|\xi|>R}|\wh H(\xi)|^2\,d\xi\le R^{-2m}\int_{\R^n}(1+|\xi|^2)^m|\wh H(\xi)|^2\,d\xi.
		\]  
		Plancherel's theorem gives \eqref{eq:Sobolev-tail}.
	\end{proof}
	
	\subsubsection{A Fourier proof of the interpolation inequality}
	
	\begin{lemma}[Fourier Gagliardo--Nirenberg estimate]\label{lem:GN}
		Let $m>n/2$. There is $C=C(n,m)>0$ such that every scalar function $F\in H^m(\R^n)$ satisfies
		\begin{equation}\label{eq:GN}
			\|F\|_{L^\infty}\le C\|F\|_{L^2}^{1-\frac{n}{2m}}\|F\|_{H^m}^{\frac{n}{2m}}.
		\end{equation}
		The same estimate holds for every matrix field $F=(F_{jk})\in H^m(\R^n;\C^{n\times n})$, with the matrix norms fixed above.
	\end{lemma}
	
	\begin{proof}
		We prove the scalar and matrix-valued estimates at the same time. In the matrix-valued case, every occurrence of $|\wh F(\xi)|$ below denotes $(\sum_{j,k=1}^n|\wh F_{jk}(\xi)|^2)^{1/2}$. Set $X_F=\|F\|_{L^2}$ and $M_F=\|F\|_{H^m}$. If $X_F=0$, then $F=0$ almost everywhere. We assume that $X_F>0$. Since $M_F\ge X_F$, the number $R=(M_F/X_F)^{1/m}$ satisfies $R\ge1$.
		
		In the scalar case, the Fourier inversion formula gives $\|F\|_{L^\infty}\le C\int_{\R^n}|\wh F(\xi)|\,d\xi$. For a matrix field, Fourier inversion is applied entrywise. For every complex matrix $M$, one has $\sup_{|v|=1}|Mv|\le |M|$. Applying Fourier inversion entrywise and using this inequality gives the same bound for matrix fields.
		
		On the region $|\xi|\le R$, the Cauchy--Schwarz inequality gives $\int_{|\xi|\le R}|\wh F(\xi)|\,d\xi\le CR^{n/2}X_F$. On the region $|\xi|>R$, the weighted Cauchy--Schwarz inequality gives
		\[
		\int_{|\xi|>R}|\wh F(\xi)|\,d\xi\le M_F\bigg(\int_{|\xi|>R}(1+|\xi|^2)^{-m}\,d\xi\bigg)^{1/2}\le CM_FR^{n/2-m},
		\] 
		where the last estimate uses $m>n/2$.
		
		For $R=(M_F/X_F)^{1/m}$, one has $R^{n/2}X_F=X_F^{1-\frac{n}{2m}}M_F^{\frac{n}{2m}}$ and $M_FR^{n/2-m}=X_F^{1-\frac{n}{2m}}M_F^{\frac{n}{2m}}$. Combining the low- and high-frequency estimates proves \eqref{eq:GN}.
	\end{proof}
	
	\subsubsection{Proof of the conductivity theorem}
	
	The Sobolev order $s_*=n/2+1$ comes from the factor $R^{(n+2)/2}$ in the low-frequency estimate. The quadratic dependence on $X$ allows us to choose the splitting radius in terms of $X$, rather than $\delta$. At the order $s_*$, the Sobolev tail then becomes a constant multiple of $MX$, while Gagliardo--Nirenberg interpolation places the chosen radius in the available frequency range. This identifies the regularity used in the argument; it does not assert that $s_*$ is optimal for the inverse problem.
	
	\begin{proof}[Proof of Theorem~\ref{thm:conductivity-main}]
		Set $X=\|H\|_{L^2(\R^n)}$, $\delta=\|H\|_{L^\infty}$, and $M=\|H\|_{H^{s_*}(\R^n)}$. If $X=0$ or $\delta=0$, then $H=0$ almost everywhere. We assume that $X>0$ and $\delta>0$.
		
		Since $s_*>n/2$, applying the Sobolev embedding entrywise gives $\delta\le C_SM$ for a constant $C_S>0$. All constants below depend only on $n$, $\Omega$, and the fixed set $K$. We choose $\varepsilon_0$ at the end, after these constants have been fixed, and require $C_S\varepsilon_0\le\delta_0$. Under the assumption $M<\varepsilon_0$, Proposition~\ref{prop:corrector-estimate}, Proposition~\ref{prop:Fourier-estimate}, and Lemma~\ref{lem:low-frequency} are then available.
		
		For every $1\le R\le c_1\delta^{-1}$, Plancherel's theorem, Lemma~\ref{lem:low-frequency}, and Lemma~\ref{lem:Sobolev-tail} with $m=s_*$ give
		\begin{equation}\label{eq:main-frequency-split}
			X\le C_0R^{s_*}X^2+C_0R^{-s_*}M
		\end{equation}
		for a constant $C_0>0$. Fix $\kappa>0$ so small that $C_0\kappa^{s_*}\le1/2$, and set $R=\kappa X^{-1/s_*}$. We verify that this choice lies in the required range. Since $X\le M$, the condition $M<\varepsilon_0\le\kappa^{s_*}$ gives $R\ge1$. Lemma~\ref{lem:GN} with $m=s_*$ gives
		\begin{equation}\label{eq:main-delta-interpolation}
			\delta\le C_2X^{1-\frac{n}{2s_*}}M^{\frac{n}{2s_*}}=C_2X^{1/s_*}M^{1-1/s_*},
		\end{equation}
		where we used $2s_*=n+2$. Multiplication by $R=\kappa X^{-1/s_*}$ yields $R\delta\le C_2\kappa M^{1-1/s_*}$. Since $1-1/s_*=n/(n+2)>0$, the inequality $R\le c_1\delta^{-1}$ holds once $\varepsilon_0$ is chosen so that $C_2\kappa\varepsilon_0^{1-1/s_*}\le c_1$.
		
		For this value of $R$, the first term on the right-hand side of \eqref{eq:main-frequency-split} is $C_0\kappa^{s_*}X\le X/2$. Moving this term to the left-hand side of \eqref{eq:main-frequency-split} gives
		\begin{equation}\label{eq:main-X-intermediate}
			X\le 2C_0R^{-s_*}M=2C_0\kappa^{-s_*}MX.
		\end{equation}
		The constants $C_S$, $\delta_0$, $C_0$, $\kappa$, $C_2$, and $c_1$ are independent of $H$. Choose $\varepsilon_0>0$ so small that $C_S\varepsilon_0\le\delta_0$, $\varepsilon_0\le\kappa^{s_*}$, $C_2\kappa\varepsilon_0^{1-1/s_*}\le c_1$, and $2C_0\kappa^{-s_*}\varepsilon_0<1$. If $M<\varepsilon_0$, all the estimates above are valid. Since $X>0$, division of \eqref{eq:main-X-intermediate} by $X$ gives $1\le 2C_0\kappa^{-s_*}M<2C_0\kappa^{-s_*}\varepsilon_0<1$, which is impossible. We conclude that $X=0$, and hence $H=0$.
	\end{proof}

	\section{Metric rigidity in dimensions \texorpdfstring{$n\ge3$}{n>=3}}\label{sec:metric}
	
	The metric reduction uses two different boundary fixing diffeomorphisms, and their roles are distinct. The first is the canonical harmonic map $Y$, obtained by solving a Dirichlet problem for each Cartesian coordinate function. The smallness of $g-e$ makes $Y$ a global diffeomorphism of $\ol\Omega$. Its global invertibility allows us to use it as a change of coordinates on the whole domain, while the harmonicity of its components forces the transformed conductivity to be divergence free.
	
	The second diffeomorphism, denoted by $F_0$, is obtained later from smooth boundary determination. Its only role is to put the metric in a representative whose full boundary jet agrees with the Euclidean metric. The composition $Z=Y\circ F_0$ compares this boundary normalized representative with the canonical harmonic gauge. The conductivity theorem will be applied to the perturbation associated with $Y_*g$, not to the metric obtained from $F_0$.
	
	\subsection{The canonical global harmonic gauge}
	
	Let $g$ be a smooth Riemannian metric on $\ol\Omega$ and set
	$A_g=(\det g)^{1/2}g^{-1}$. For every smooth function $u$,
	$-\diver(A_g\nabla u)=(\det g)^{1/2}(-\Delta_g u)$, so the equations
	$-\Delta_gu=0$ and $-\diver(A_g\nabla u)=0$ are equivalent.
	
	We use $\Lambda_g$ for the geometric DN map. If $u_f$ is $g$-harmonic with boundary value $f$ and $V$ has boundary trace $\psi$, Green's formula gives
	\begin{equation}\label{eq:metric-conductivity-Green}
		\int_{\p\Omega}(\Lambda_gf)\psi\,dS_g
		=\int_\Omega A_g\nabla u_f\cdot\nabla V\,dx.
	\end{equation}
	In particular, when the induced boundary metric is Euclidean, the geometric DN map agrees with the conductivity DN map associated with $A_g$.
	
	For $1\le j\le n$, recall that $y^j$ is the unique solution of \eqref{eq:BVP-harmonic-coord}, and set $Y=(y^1,\ldots,y^n)$. At this stage, $Y$ is only a smooth map from $\ol\Omega$ to $\R^n$. The next proposition shows that the near-Euclidean assumption makes this map a global boundary fixing diffeomorphism. Only after this fact has been proved do we use $Y$ as a change of coordinates.
	
	\begin{proposition}[Canonical global harmonic gauge]\label{prop:harmonic-coordinates}
		Let $n\ge3$ and let $s_*-k<\alpha<1$. There are constants $\varepsilon,C>0$, depending only on $n$, $\Omega$, and $\alpha$, with the following property. Suppose that $g\in C^\infty(\ol\Omega;\operatorname{Sym}^2\R^n)$ is a Riemannian metric satisfying $\|g-e\|_{C^{k,\alpha}(\ol\Omega)}<\varepsilon$. Let $Y=(y^1,\ldots,y^n)$ be the map defined by \eqref{eq:BVP-harmonic-coord}. Then $Y\in\Diff^\infty(\ol\Omega)$, $Y|_{\p\Omega}=\Id$, and
		\begin{equation}\label{eq:Y-Schauder}
			\|Y-\Id\|_{C^{k+1,\alpha}(\ol\Omega)}
			\le C\|g-e\|_{C^{k,\alpha}(\ol\Omega)}.
		\end{equation}
		Set $\widetilde g=Y_*g=(Y^{-1})^*g$ and $\widetilde A=(\det\widetilde g)^{1/2}\widetilde g^{-1}=Y_*A_g=I+H$. Then $\Lambda_{\widetilde g}=\Lambda_g$, $\diver H=0$ in $\Omega$, and
		\begin{equation}\label{eq:tilde-A-smallness}
			\|H\|_{C^{k,\alpha}(\ol\Omega)}
			\le C\|g-e\|_{C^{k,\alpha}(\ol\Omega)}.
		\end{equation}
	\end{proposition}

	\begin{proof}
		Let $e_j$ denote the $j$th standard basis vector of $\R^n$, and set $v^j=y^j-x^j$ and $v=(v^1,\ldots,v^n)$. The boundary conditions for $y^j$ give $v|_{\p\Omega}=0$ and $Y=\Id+v$. Since $\diver(A_g\nabla y^j)=0$ and $\nabla x^j=e_j$, the function $v^j\in H_0^1(\Omega)$ satisfies $\diver(A_g\nabla v^j)=-\diver((A_g-I)e_j)$ in $\Omega$.
		
		In the fixed Euclidean coordinates, let $\mathcal P\subset\operatorname{Sym}^2\R^n$ be the open set of positive definite symmetric matrices, and define $\mathcal C:\mathcal P\to\operatorname{Sym}^2\R^n$ by $\mathcal C(q)=(\det q)^{1/2}q^{-1}$. Then $A_g(x)=\mathcal C(g(x))$ for every $x\in\ol\Omega$, while $\mathcal C(e)=I$.
		
		After decreasing $\varepsilon$, every matrix $e+\tau(g(x)-e)$, with $x\in\ol\Omega$ and $0\le\tau\le1$, has all its eigenvalues in the fixed interval $[1/2,3/2]$. These matrices belong to a fixed compact subset $\mathcal K\Subset\mathcal P$. For $q\in\mathcal P$ and $S\in\operatorname{Sym}^2\R^n$, direct differentiation gives $D\mathcal C(q)[S]=(\det q)^{1/2}\big(\frac12\operatorname{tr}(q^{-1}S)q^{-1}-q^{-1}Sq^{-1}\big)$. Since $\mathcal C$ is smooth on $\mathcal P$, its derivatives of every fixed finite order are uniformly bounded on $\mathcal K$. The fundamental theorem of calculus, applied pointwise in $x$, gives $A_g(x)-I=\int_0^1D\mathcal C\bigl(e+\tau(g(x)-e)\bigr)[g(x)-e]\,d\tau$. After differentiating this identity in $x$ up to order $k$, every resulting term is a product of derivatives of $g-e$ and a derivative of $\mathcal C$ evaluated at $e+\tau(g-e)$. Every term contains at least one factor involving a derivative of $g-e$ of order between zero and $k$. The uniform bounds for the derivatives of $\mathcal C$ on $\mathcal K$, the product estimates in $C^{k,\alpha}(\ol\Omega)$, and the corresponding H\"older seminorm estimates give
		\begin{equation}\label{eq:Ag-smallness}
			\|A_g-I\|_{C^{k,\alpha}(\ol\Omega)}
			\le C\|g-e\|_{C^{k,\alpha}(\ol\Omega)}.
		\end{equation}
		The same choice of $\varepsilon$ places the eigenvalues of $A_g(x)$ between two fixed positive constants and gives a uniform bound for $\|A_g\|_{C^{k,\alpha}(\ol\Omega)}$. The constants in the global boundary Schauder and $C^0$ estimates below may be chosen uniformly for all metrics satisfying the stated smallness condition.
		
		Expanding the equation for $v^j$ gives $A_g^{i\ell}\p_i\p_\ell v^j+(\p_iA_g^{i\ell})\p_\ell v^j=-\p_i(A_g^{ij}-\delta^{ij})$, where $\delta^{ij}$ are the entries of the identity matrix. The leading coefficients belong to $C^{k,\alpha}(\ol\Omega)$, the first order coefficients belong to $C^{k-1,\alpha}(\ol\Omega)$, and the right-hand side belongs to $C^{k-1,\alpha}(\ol\Omega)$. The global boundary Schauder estimate for the zero Dirichlet problem, together with the uniform $C^0$ estimate, gives
		\begin{equation}\label{eq:estimate-of-v}
			\|v^j\|_{C^{k+1,\alpha}(\ol\Omega)}
			\le C\|\p_i(A_g^{ij}-\delta^{ij})\|_{C^{k-1,\alpha}(\ol\Omega)}
			\le C\|A_g-I\|_{C^{k,\alpha}(\ol\Omega)}.
		\end{equation}
		Summing over $1\le j\le n$ and using \eqref{eq:Ag-smallness} proves \eqref{eq:Y-Schauder}. Since $g$ and the boundary data are smooth, smooth elliptic regularity gives $Y\in C^\infty(\ol\Omega)$.
		
		We next prove that $Y$ is a diffeomorphism of $\ol\Omega$. Since $\p\Omega$ is smooth, the half-space extension construction, applied in boundary charts and combined with a smooth partition of unity, gives a linear extension operator that is bounded on H\"older spaces and preserves smoothness \cite{Seeley1964}. Applied componentwise, it gives a bounded extension operator $E:C^{k+1,\alpha}(\ol\Omega;\R^n)\to C^{k+1,\alpha}(\R^n;\R^n)$ that maps smooth functions to smooth functions. Fix this operator independently of $g$. Choose $\chi\in C_c^\infty(\R^n)$ such that $\chi=1$ in a neighborhood of $\ol\Omega$, and set $\widetilde v=\chi Ev$. Then $\widetilde v=v$ on $\ol\Omega$, $\widetilde v$ has compact support, and $\|\widetilde v\|_{C^{k+1,\alpha}(\R^n)}\le C\|v\|_{C^{k+1,\alpha}(\ol\Omega)}$, where $C$ depends only on $n$, $\Omega$, $\alpha$, and the fixed cutoff.
		
		For the derivative matrix, we use $\|D\widetilde v\|_{L^\infty(\R^n)}=\sup_{x\in\R^n}
		\sup_{|\xi|=1}|D\widetilde v(x)\xi|$. After decreasing $\varepsilon$, estimates \eqref{eq:Y-Schauder}, \eqref{eq:Ag-smallness} and \eqref{eq:estimate-of-v} give $\|D\widetilde v\|_{L^\infty(\R^n)}<1/2$. For $0\le s\le1$, define
		\begin{equation}\label{eq:def-tilde-Y}
			\widetilde Y_s(x)=x+s\widetilde v(x).
		\end{equation}
		The line segment formula gives, for every $x,x'\in\R^n$,
		\[
		|\widetilde Y_s(x)-\widetilde Y_s(x')|\ge \bigl(1-s\|D\widetilde v\|_{L^\infty(\R^n)}\bigr)|x-x'|
		\ge\frac12|x-x'|.
		\]
		It follows that $\widetilde Y_s$ is injective. Its derivative is $D\widetilde Y_s=I+sD\widetilde v$. If
		$(I+sD\widetilde v(x))\xi=0$, then $|\xi|\le s\|D\widetilde v\|_{L^\infty}|\xi|<|\xi|$ unless $\xi=0$. Thus, $D\widetilde Y_s(x)$ is invertible for every $x$, and $\widetilde Y_s$ is a local diffeomorphism.
		
		Since $\widetilde v$ has compact support, $\widetilde Y_s$ agrees with the identity outside the compact set $K_0=\supp\widetilde v$. For every compact set $L\subset\R^n$, the preimage $\widetilde Y_s^{-1}(L)$ is closed and contained in the compact set $K_0\cup L$, so it is compact. Thus, $\widetilde Y_s$ is proper. Its image is open because $\widetilde Y_s$ is a local diffeomorphism, and it is closed because $\widetilde Y_s$ is proper. The image is nonempty, and $\R^n$ is connected, so $\widetilde Y_s(\R^n)=\R^n$. It follows that $\widetilde Y_s$ is a global diffeomorphism of $\R^n$. The determinant $\det D\widetilde Y_s(x)$ is nonzero and depends continuously on $s$. Since $\det D\widetilde Y_0(x)=1$, one has $\det D\widetilde Y_s(x)>0$ for every $x\in\R^n$ and every $0\le s\le1$.
		
		Using \eqref{eq:def-tilde-Y}, the boundary condition $\widetilde v|_{\p\Omega}=v|_{\p\Omega}=0$ gives $\widetilde Y_s(p)=p$ for every $p\in\p\Omega$. Let $x\in\R^n\setminus\p\Omega$. The path $s\mapsto\widetilde Y_s(x)$ cannot meet $\p\Omega$. Indeed, if $\widetilde Y_s(x)=p\in\p\Omega$ for some $s$, then $\widetilde Y_s(p)=p$, and injectivity gives $x=p$, which is impossible. Since $\widetilde Y_0(x)=x$, this path stays in the connected component of $\R^n\setminus\p\Omega$ containing $x$. The domain $\Omega$ is connected and is both open and closed in $\R^n\setminus\p\Omega$, so it is one such component. For every $x\in\Omega$ and $0\le s\le1$, we have $\widetilde Y_s(x)\in\Omega$. This proves $\widetilde Y_s(\Omega)\subset\Omega$. Conversely, let $y\in\Omega$ and set $x=\widetilde Y_s^{-1}(y)$. The point $x$ cannot lie on $\p\Omega$, because $\widetilde Y_s$ fixes the boundary. The path $r\mapsto\widetilde Y_r(x)$ for $0\le r\le s$ avoids $\p\Omega$ and joins $x$ to $y$, so $x$ and $y$ belong to the same connected component of $\R^n\setminus\p\Omega$. Since $y\in\Omega$, we obtain $x\in\Omega$, and this proves $\widetilde Y_s(\Omega)=\Omega$. At $s=1$, the restriction of $\widetilde Y_1$ to $\ol\Omega$ is $Y=\Id+v$. The global inverse $\widetilde Y_1^{-1}$ is smooth on $\R^n$, so its restriction is smooth on $\ol\Omega$. We have proved that $Y\in\Diff^\infty(\ol\Omega)$, that $Y|_{\p\Omega}=\Id$, and that $Y$ is orientation preserving.
		
		We may now use $Y$ as a global change of coordinates. Write $y=Y(x)$ and set $\widetilde g=Y_*g=(Y^{-1})^*g$. Since $Y$ fixes $\p\Omega$ pointwise, invariance of the geometric DN map gives $\Lambda_{\widetilde g}=\Lambda_g$. At $y=Y(x)$, the transformed metric satisfies
		\[
		\widetilde g(y)=DY(x)^{-T}g(x)DY(x)^{-1},\quad \widetilde g^{-1}(y)=DY(x)g^{-1}(x)DY(x)^T, \quad (\det\widetilde g(y))^{1/2}=\frac{(\det g(x))^{1/2}}{\det DY(x)}.
		\]
		Since $\det DY(x)>0$, it follows that
		\[
		A_{\widetilde g}(y)=(\det\widetilde g(y))^{1/2}\widetilde g^{-1}(y)=\frac{DY(x)A_g(x)DY(x)^T}{\det DY(x)}=(Y_*A_g)(y).
		\]
		Set $\widetilde A=A_{\widetilde g}=Y_*A_g$.
		
		We next prove the harmonic gauge condition. For $1\le j\le n$, let $\pi_j(y)=y^j$ be the $j$th Cartesian coordinate function in the target variables. The identity $\pi_j\circ Y=y^j$ shows that $\pi_j$ is the transform of the $g$-harmonic function $y^j$. Let $\phi\in C_c^\infty(\Omega)$. Since $y^j$ is $A_g$-harmonic, the function $\phi\circ Y$ may be used as a test function. Using
		$\nabla_x y^j=DY(x)^Te_j$, $\nabla_x(\phi\circ Y)=DY(x)^T\nabla_y\phi(Y(x))$, and the change of variables $y=Y(x)$, we obtain
		\begin{equation*}
			\begin{aligned}
				0=\int_\Omega A_g(x)\nabla_x y^j(x)\cdot\nabla_x(\phi\circ Y)(x)\,dx=\int_\Omega
				\widetilde A(y)e_j\cdot\nabla_y\phi(y)\,dy=\int_\Omega \widetilde A^{ij}(y)\p_{y_i}\phi(y)\,dy.
			\end{aligned}
		\end{equation*}
		This identity holds for every $\phi\in C_c^\infty(\Omega)$, so $\p_{y_i}\widetilde A^{ij}=0$ in distributions for every $j$. Since $\widetilde A$ is smooth, the identities hold classically. Writing $\widetilde A=I+H$ gives $\diver H=0$.
		
		It remains to prove \eqref{eq:tilde-A-smallness}. The estimate for $Y^{-1}$ is needed because the pushforward coefficient is expressed in the $y$-variable and involves composition with $Y^{-1}$. Write $Y=\Id+v$. Then $\Id=Y^{-1}+v\circ Y^{-1}$, and hence $Y^{-1}-\Id=-v\circ Y^{-1}$. Differentiating $Y\circ Y^{-1}=\Id$ gives $DY^{-1}(y)=\bigl[DY(Y^{-1}(y))\bigr]^{-1}$. Since $\|DY-I\|_{L^\infty(\Omega)}<1/2$, for every $x\in\Omega$ and $\xi\in\R^n$, one has $|DY(x)\xi|\ge\bigl(1-\|DY-I\|_{L^\infty(\Omega)}\bigr)|\xi|>\frac12|\xi|$. It follows that $DY(x)$ is invertible and $|[DY(x)]^{-1}\eta|\le2|\eta|$ for every $\eta\in\R^n$, uniformly in $x$. In particular, $DY^{-1}$ is uniformly bounded. Moreover, $DY^{-1}-I=-\bigl[(DY)\circ Y^{-1}\bigr]^{-1}\bigl[(DY-I)\circ Y^{-1}\bigr]$, so the first derivative of $Y^{-1}-\Id$ is controlled linearly by $DY-I$. The global lower Lipschitz bound for $\widetilde Y_1$ also gives $|\widetilde Y_1^{-1}(y)-\widetilde Y_1^{-1}(y')|\le2|y-y'|$ for all $y,y'\in\R^n$. Thus, its restriction $Y^{-1}$ is uniformly Lipschitz on $\ol\Omega$.
		
		Put $a=\|v\|_{C^{k+1,\alpha}(\ol\Omega)}$ and assume that $a\le a_0$, where $a_0\le1$ is fixed and sufficiently small. The fixed extension estimate for $v$ and the uniform Lipschitz bound for $Y^{-1}$ give $\|(D^rv)\circ Y^{-1}\|_{C^\alpha(\ol\Omega)}\le Ca$ for $1\le r\le k+1$. To start the induction in the H\"older norm, set $B(y)=DY(Y^{-1}(y))$. The matrix identity $B(y)^{-1}-B(z)^{-1}=B(y)^{-1}(B(z)-B(y))B(z)^{-1}$ and the bound $\|B^{-1}\|_{L^\infty}\le2$ give $[DY^{-1}]_{C^\alpha}\le4[Dv\circ Y^{-1}]_{C^\alpha}\le Ca$. Together with the preceding estimate for $DY^{-1}-I$, this yields $\|DY^{-1}-I\|_{C^\alpha}\le Ca$ and $\|DY^{-1}\|_{C^\alpha}\le C$.
		
		For $2\le\ell\le k+1$, the higher order chain rule applied to $Y\circ Y^{-1}=\Id$ gives $(DY\circ Y^{-1})D^\ell Y^{-1}+\mathcal R_\ell=0$. Each term in $\mathcal R_\ell$ is a contraction of $(D^rv)\circ Y^{-1}$ with $D^{k_1}Y^{-1},\ldots,D^{k_r}Y^{-1}$, where $2\le r\le\ell$, $k_i\ge1$, and $k_1+\cdots+k_r=\ell$. In particular, every $k_i$ is at most $\ell-1$. Suppose inductively that the $C^\alpha$ norms of the derivatives of $Y^{-1}$ through order $\ell-1$ are uniformly bounded. The H\"older product estimate then gives $\|\mathcal R_\ell\|_{C^\alpha}\le C_\ell a$: if the H\"older difference falls on $(D^rv)\circ Y^{-1}$, its seminorm contributes the factor $a$; if it falls on another factor, the $L^\infty$ norm of $(D^rv)\circ Y^{-1}$ still contributes this factor. Since $(DY\circ Y^{-1})^{-1}=DY^{-1}$, we obtain $D^\ell Y^{-1}=-DY^{-1}\mathcal R_\ell$ and $\|D^\ell Y^{-1}\|_{C^\alpha}\le C_\ell a$. This completes the induction, including the H\"older estimate at order $k+1$.
		
		Finally, $Y^{-1}-\Id=-v\circ Y^{-1}$ gives $\|Y^{-1}-\Id\|_{L^\infty}\le\|v\|_{L^\infty}$. Combining this with the first derivative estimate and the estimates above gives
		\begin{equation}\label{eq:Y-inverse-Schauder}
			\|Y^{-1}-\Id\|_{C^{k+1,\alpha}(\ol\Omega)}\le C\|Y-\Id\|_{C^{k+1,\alpha}(\ol\Omega)}.
		\end{equation}
		The constant is uniform when $\|Y-\Id\|_{C^{k+1,\alpha}(\ol\Omega)}$ is sufficiently small.
		
		Define $\mathcal T(P,Q)=\frac{PQP^T}{\det P}$ for matrices $P$ with positive determinant and symmetric positive definite matrices $Q$. The map $\mathcal T$ is smooth in a neighborhood of $(I,I)$ and satisfies $\mathcal T(I,I)=I$. The pushforward formula is $\widetilde A=\mathcal T(DY,A_g)\circ Y^{-1}$. By \eqref{eq:Y-inverse-Schauder}, the maps $Y^{-1}$ remain in a fixed bounded subset of $C^{k+1,\alpha}(\ol\Omega)$, and composition with $Y^{-1}$ is uniformly bounded on $C^{k,\alpha}(\ol\Omega)$. Since $C^{k,\alpha}(\ol\Omega)$ is a Banach algebra and $\mathcal T$ is smooth near $(I,I)$, we obtain
		\begin{equation*}
			\begin{aligned}
				\|H\|_{C^{k,\alpha}(\ol\Omega)}
				&=\|\widetilde A-I\|_{C^{k,\alpha}(\ol\Omega)}
				\le C\|\mathcal T(DY,A_g)-I\|_{C^{k,\alpha}(\ol\Omega)}\\
				&\le C\bigl(\|DY-I\|_{C^{k,\alpha}(\ol\Omega)}+\|A_g-I\|_{C^{k,\alpha}(\ol\Omega)}\bigr)
				\le C\|g-e\|_{C^{k,\alpha}(\ol\Omega)},
			\end{aligned}
		\end{equation*}
		where the last inequality follows from \eqref{eq:Y-Schauder} and \eqref{eq:Ag-smallness}. This proves \eqref{eq:tilde-A-smallness}.
	\end{proof}

	From now on, $Y$ denotes the fixed global diffeomorphism constructed in Proposition~\ref{prop:harmonic-coordinates}. It is not an arbitrary coordinate change. Its global invertibility allows the metric and conductivity to be transformed on the whole domain, while the harmonicity of its components gives the divergence free condition for $H$. The perturbation used in the remainder of the metric argument is always the tensor $H=A_{Y_*g}-I$ associated with this canonical map.
	
	\subsection{Boundary determination and comparison of the two gauges}
	
	Boundary determination provides a second boundary fixing diffeomorphism used only to normalize the full boundary jet. We record the required statement. For a smooth tensor $T$, the notation $j^\infty_{\p\Omega}T=0$ means that all derivatives of its coordinate coefficients vanish on $\p\Omega$ in every boundary chart, or equivalently that $T$ vanishes to infinite order at the boundary.
	
	\begin{proposition}[Boundary determination]\label{prop:boundary-determination-zero}
		Let $g\in C^\infty(\ol\Omega;\operatorname{Sym}^2\R^n)$ be a Riemannian metric and suppose that \eqref{eq:same-DN-map} holds. Then $g$ and $e$ induce the same metric on $\p\Omega$, meaning that $\xi^Tg(x)\eta=\xi\cdot\eta$ for every $x\in\p\Omega$ and all $\xi,\eta\in\R^n$ satisfying $\xi\cdot\nu(x)=\eta\cdot\nu(x)=0$. Moreover, there exists $F_0\in\Diff^\infty(\ol\Omega)$ such that $F_0|_{\p\Omega}=\Id$ and
		\begin{equation}\label{eq:boundary-determination-zero}
			j^\infty_{\p\Omega}(F_0^*g-e)=0.
		\end{equation}
	\end{proposition}
	
\begin{proof}
	The principal symbol calculation for the DN map \cite[Proposition~1.3]{LeeUhlmann1989} shows that \eqref{eq:same-DN-map} determines the induced boundary metric. In particular, $g$ and $e$ induce the same metric on $\p\Omega$.
		
	For $0\le s\le1$, set $g_s=(1-s)e+sg$. The matrix $g_s(x)$ is positive definite for every $x\in\ol\Omega$, so $(g_s)_{0\le s\le1}$ is a smooth family of Riemannian metrics. Let $\nu$ be the outward unit normal on $\p\Omega$ with respect to $e$, and let $\nu_s$ be the outward unit normal with respect to $g_s$. In the fixed Euclidean coordinates, $\nu_s(p)=g_s(p)^{-1}\nu(p)/\sqrt{\nu(p)^Tg_s(p)^{-1}\nu(p)}$. Define $\Phi_s(p,t)=\exp_p^{g_s}(-t\nu_s(p))$. Thus, $\Phi_s(p,t)$ is the point reached at time $t$ by the $g_s$ unit speed geodesic starting from $p$ in the inward normal direction. The vector field $\nu_s$ and the geodesic equation depend smoothly on $(s,p)$. Compactness of $[0,1]\times\p\Omega$ allows us to choose $\varepsilon>0$ such that every $\Phi_s$ is a smooth embedding of $\p\Omega\times[0,\varepsilon)$ onto a neighborhood of $\p\Omega$ in $\ol\Omega$, and the map $(s,p,t)\mapsto\Phi_s(p,t)$ is smooth. Write $\Phi_e=\Phi_0$ and $\Phi_g=\Phi_1$.
		
	In these normal parametrizations, the endpoint metrics have the forms $\Phi_e^*e=dt^2+h_e(t)$ and $\Phi_g^*g=dt^2+h_g(t)$, where $h_e(t)$ and $h_g(t)$ are smooth families of metrics on $\p\Omega$. Equality of the DN maps implies equality of their full symbols. For scalar functions, the boundary normal recursion in \cite[Theorem~1.1(ii)]{JoshiLionheart2005} gives $\p_t^m h_g(0)=\p_t^m h_e(0)$ for every $m\ge0$. These are equalities of smooth symmetric tensors on $\p\Omega$.
		
	Set $U_e=\Phi_e(\p\Omega\times[0,\varepsilon))$ and define $F_{\mathrm{col}}:U_e\to\ol\Omega$ by $F_{\mathrm{col}}(\Phi_e(p,t))=\Phi_g(p,t)$. Equivalently, $F_{\mathrm{col}}=\Phi_g\circ\Phi_e^{-1}$. Since $\Phi_e(p,0)=\Phi_g(p,0)=p$, the map $F_{\mathrm{col}}$ fixes $\p\Omega$ pointwise. Moreover, $\Phi_e^*(F_{\mathrm{col}}^*g-e)=h_g(t)-h_e(t)$. Every normal derivative of the right-hand side vanishes at $t=0$. Since each such normal derivative is the zero tensor field on $\p\Omega$, all of its tangential derivatives vanish as well. It follows that
	\begin{equation}\label{eq:Fcol-boundary-flat}
		j^\infty_{\p\Omega}(F_{\mathrm{col}}^*g-e)=0.
	\end{equation}
		
	Choose $\chi\in C^\infty([0,\varepsilon))$ equal to $1$ for $0\le t\le\varepsilon/2$ and to $0$ for $3\varepsilon/4\le t<\varepsilon$. Set $U_s=\Phi_s(\p\Omega\times[0,\varepsilon))$, define $V_s(\Phi_s(p,t))=\chi(t)\p_s\Phi_s(p,t)$ on $U_s$, and extend $V_s$ by zero to $\ol\Omega\setminus U_s$. Since $\chi$ vanishes near the inner edge of the collar, this defines a smooth time dependent vector field on $\ol\Omega$. Moreover, $\Phi_s(p,0)=p$ is independent of $s$, so $V_s(p)=\p_s\Phi_s(p,0)=0$ for every $p\in\p\Omega$.
		
	For each $x\in\ol\Omega$, define $\widetilde f_s(x)$ by the initial value problem
	\begin{equation}\label{eq:boundary-gauge-flow}
		\begin{cases}
			\p_s\widetilde f_s(x)=V_s(\widetilde f_s(x)), & 0\le s\le1,\\
			\widetilde f_0(x)=x.
		\end{cases}
	\end{equation}
	Boundary points give constant solutions, so uniqueness implies $\widetilde f_s(p)=p$ for every $p\in\p\Omega$. The same uniqueness, applied backward in time, prevents an interior trajectory from reaching the boundary in finite time. Smoothness of $V_s$ and compactness of $\ol\Omega$ give solutions for all $0\le s\le1$, with smooth dependence on $(s,x)$. Moreover, for each fixed $s$, solving $\p_r\gamma(r)=V_r(\gamma(r))$ backward from $\gamma(s)=y$ gives the smooth inverse $\widetilde f_s^{-1}(y)=\gamma(0)$. Thus, $\widetilde f_s\in\Diff^\infty(\ol\Omega)$ and $\widetilde f_s|_{\p\Omega}=\Id$.
		
	Fix $p\in\p\Omega$ and $0\le t\le\varepsilon/2$. Since $\chi(t)=1$, the curve $\gamma(s)=\Phi_s(p,t)$ satisfies $\p_s\gamma(s)=V_s(\gamma(s))$ and $\gamma(0)=\Phi_e(p,t)$. Uniqueness in \eqref{eq:boundary-gauge-flow} gives $\widetilde f_s(\Phi_e(p,t))=\Phi_s(p,t)$. In particular, $F_0=\widetilde f_1$ satisfies $F_0(\Phi_e(p,t))=\Phi_g(p,t)$ for $0\le t\le\varepsilon/2$, so it agrees with $F_{\mathrm{col}}$ near $\p\Omega$. Formula~\eqref{eq:Fcol-boundary-flat} now proves \eqref{eq:boundary-determination-zero}.
\end{proof}
	
The conductivity associated with $F_0^*g$ need not be divergence free. We use this auxiliary gauge only to prove boundary flatness in the canonical harmonic coordinates, without requiring any quantitative estimate for $F_0$.

\begin{lemma}[Boundary flatness in the canonical harmonic gauge]\label{lem:harmonic-boundary-flatness}
	Assume the hypotheses of Proposition~\ref{prop:harmonic-coordinates} and suppose that $\Lambda_g=\Lambda_e$. Let $Y$ be the canonical harmonic map constructed there, set $\widetilde g=Y_*g$, and write $A_{\widetilde g}=I+H$. Then
	\begin{equation}\label{eq:boundary-flatness}
		j^\infty_{\p\Omega}(\widetilde g-e)=0,\quad j^\infty_{\p\Omega}H=0.
	\end{equation}
\end{lemma}

\begin{proof}
	Let $F_0$ be given by Proposition~\ref{prop:boundary-determination-zero}, and set $g_0=F_0^*g$ and $Z=Y\circ F_0$. Then $\Lambda_{g_0}=\Lambda_e$ and
	\begin{equation}\label{eq:g0-boundary-flat}
		j^\infty_{\p\Omega}(g_0-e)=0.
	\end{equation}
	The change of coordinates gives $\Delta_{g_0}Z^j=0$ in $\Omega$ and $Z^j=x^j$ on $\p\Omega$. Since $g_0=e$ on the boundary, their unit normals agree there. Thus, equality of the DN maps gives $\p_{\nu_{g_0}}Z^j=\Lambda_{g_0}(x^j|_{\p\Omega})=\Lambda_e(x^j|_{\p\Omega})=\p_{\nu_{g_0}}x^j$. This implies that $w^j=Z^j-x^j$ has zero Dirichlet and Neumann data and satisfies
	\begin{equation}\label{eq:w-flat-source}
		\Delta_{g_0}w^j=f^j,\quad f^j=-\Delta_{g_0}x^j,\quad j^\infty_{\p\Omega}f^j=0,
	\end{equation}
	where the flatness of $f^j$ follows from \eqref{eq:g0-boundary-flat} and $\Delta_ex^j=0$.
	
	In a boundary chart $(x',t)$ with $t=0$ on $\p\Omega$, the zero Dirichlet data imply that all tangential derivatives of $w^j$ vanish there. The zero Neumann data then give $\p_tw^j=0$, since the unit normal has a nonzero component in the transverse direction $\p_t$. Let $a^{tt}>0$ be the coefficient of $\p_t^2$ in $\Delta_{g_0}$. Suppose that all derivatives of $w^j$ involving at most $m+1$ derivatives in $t$ vanish on the boundary. Applying $\p_t^m$ to \eqref{eq:w-flat-source} and restricting to $t=0$ gives $a^{tt}\p_t^{m+2}w^j=0$: all remaining terms involve at most $m+1$ derivatives in $t$ of $w^j$ or a derivative of the flat source $f^j$. Tangential differentiation and induction yield
	\begin{equation}\label{eq:Z-flat}
		j^\infty_{\p\Omega}(Z-\Id)=0,
	\end{equation}
	with the jet understood componentwise.
	
	The identity $Z^{-1}-\Id=-(Z-\Id)\circ Z^{-1}$ also gives $j^\infty_{\p\Omega}(Z^{-1}-\Id)=0$. Since $\widetilde g=Z_*g_0$, we have $\widetilde g(y)=DZ^{-1}(y)^Tg_0(Z^{-1}(y))DZ^{-1}(y)$. The chain and product rules, together with \eqref{eq:g0-boundary-flat} and \eqref{eq:Z-flat}, give $j^\infty_{\p\Omega}(\widetilde g-e)=0$. The smooth dependence of $A_{\widetilde g}=(\det\widetilde g)^{1/2}\widetilde g^{-1}$ on $\widetilde g$ then gives $j^\infty_{\p\Omega}H=0$.
\end{proof}
	
	Hence, the perturbation $H$ in the canonical harmonic gauge is divergence free, satisfies \eqref{eq:tilde-A-smallness}, and vanishes to infinite order at the boundary. The auxiliary diffeomorphism $F_0$ is not used below.
	
	\subsection{Exterior extension and transfer of the DN map}
	
	From now on, $H$ always denotes the perturbation determined by the canonical harmonic gauge $\widetilde g=Y_*g$, so that $A_{\widetilde g}=I+H$. Fix a smooth bounded connected domain $\Omega_{\mathrm e}$ satisfying $\ol\Omega\Subset\Omega_{\mathrm e}$, and define
	\begin{equation}\label{eq:He-def}
		H_{\mathrm e}(x)=
		\begin{cases}
			H(x) & \text{if }x\in\Omega,\\
			0 & \text{if }x\in\Omega_{\mathrm e}\setminus\Omega,
		\end{cases}\quad
		A_{\mathrm e}=I+H_{\mathrm e}.
	\end{equation}
	The two definitions agree on $\p\Omega$ because $H|_{\p\Omega}=0$.
	
	\begin{lemma}[Exterior extension and transfer of the DN map]\label{lem:outer-DN-transfer}
		Assume the hypotheses of Lemma~\ref{lem:harmonic-boundary-flatness}. Then
		$H_{\mathrm e}\in C_c^\infty(\Omega_{\mathrm e};\operatorname{Sym}^2\R^n)$,
		$\supp H_{\mathrm e}\subset\ol\Omega\Subset\Omega_{\mathrm e}$,
		$\diver H_{\mathrm e}=0$, and $A_{\mathrm e}$ is a real symmetric uniformly elliptic conductivity. Moreover,
		\begin{equation}\label{eq:outer-DN-equality}
			\Lambda_{A_{\mathrm e}}^{\p\Omega_{\mathrm e}}=\Lambda_I^{\p\Omega_{\mathrm e}}.
		\end{equation}
		After extending $H_{\mathrm e}$ by zero outside $\Omega_{\mathrm e}$, one also has
		\begin{equation}\label{eq:He-Hnplus1}
			\|H_{\mathrm e}\|_{H^{s_*}(\R^n)}\le C\|g-e\|_{C^{k,\alpha}(\ol\Omega)},
		\end{equation}
		where $C>0$ depends only on $n$, $\Omega$, and $\alpha$.
	\end{lemma}
	
	\begin{proof}
		Formula~\eqref{eq:boundary-flatness} shows that all derivatives of $H$ vanish on $\p\Omega$, and the piecewise definition \eqref{eq:He-def} is smooth across the interface. Since $H_{\mathrm e}=0$ on $\Omega_{\mathrm e}\setminus\Omega$, its support is contained in $\ol\Omega\Subset\Omega_{\mathrm e}$. The tensor $H_{\mathrm e}$ belongs to $C_c^\infty(\Omega_{\mathrm e};\operatorname{Sym}^2\R^n)$. The matrix $A_{\mathrm e}$ equals $A_{\widetilde g}$ in $\Omega$ and $I$ outside $\Omega$, so it is real symmetric and uniformly elliptic.
		
		We verify the divergence condition across the interface. Let $\phi\in C_c^\infty(\Omega_{\mathrm e};\R^n)$. Since $\diver H=0$ in $\Omega$ and $H=0$ on $\p\Omega$, integration by parts gives
		\begin{equation}\label{eq:He-div-proof}
			\begin{aligned}
				\int_{\Omega_{\mathrm e}}H_{\mathrm e}^{ij}\p_i\phi_j\,dx=\int_\Omega H^{ij}\p_i\phi_j\,dx=\int_{\p\Omega}H^{ij}\nu_i\phi_j\,dS-\int_\Omega(\p_iH^{ij})\phi_j\,dx
				=0.
			\end{aligned}
		\end{equation}
		This proves $\diver H_{\mathrm e}=0$ in distributions. Since $H_{\mathrm e}$ is smooth, the identity also holds classically.
		
		We next identify the conductivity DN map on the inner boundary. Since $Y$ fixes $\p\Omega$ pointwise, the invariance of the geometric DN map gives $\Lambda_{\widetilde g}=\Lambda_g=\Lambda_e$. Formula~\eqref{eq:boundary-flatness} gives $\widetilde g=e$ on $\p\Omega$, so the induced boundary measures for $\widetilde g$ and $e$ agree. Applying the weak identity \eqref{eq:metric-conductivity-Green} to $\widetilde g$ and to $e$ identifies the geometric DN maps with the conductivity DN maps associated with $I+H$ and $I$. We obtain
		\begin{equation}\label{eq:inner-conductivity-DN}
			\Lambda_{I+H}^{\p\Omega}		=\Lambda_I^{\p\Omega}.
		\end{equation}

		Fix $f\in H^{1/2}(\p\Omega_{\mathrm e})$, and let $u\in H^1(\Omega_{\mathrm e})$ be the $A_{\mathrm e}$-harmonic function with boundary trace $f$ on $\p\Omega_{\mathrm e}$. Set $q=u|_{\p\Omega}$. The restriction of $u$ to $\Omega$ is $(I+H)$-harmonic with boundary trace $q$. Let $v\in H^1(\Omega)$ be the Euclidean harmonic function with the same trace $q$, and define
		\begin{equation}\label{eq:glued-U}
			U=
			\begin{cases}
				v & \text{in }\Omega,\\
				u & \text{in }\Omega_{\mathrm e}\setminus\ol\Omega.
			\end{cases}
		\end{equation}
		The two pieces have the same trace on $\p\Omega$, so $U\in H^1(\Omega_{\mathrm e})$.
		
		Let $\varphi\in C_c^\infty(\Omega_{\mathrm e})$. Since $u$ satisfies the global weak equation in $\Omega_{\mathrm e}$ and $A_{\mathrm e}=I$ in $\Omega_{\mathrm e}\setminus\ol\Omega$, we have
		\[
		\int_{\Omega_{\mathrm e}\setminus\ol\Omega}\nabla u\cdot\nabla\varphi\,dx
		=-\int_\Omega(I+H)\nabla u\cdot\nabla\varphi\,dx
		=-\langle\Lambda_{I+H}q,\varphi|_{\p\Omega}\rangle.
		\]
		Since $v$ is Euclidean harmonic in $\Omega$, we have $\int_\Omega\nabla v\cdot\nabla\varphi\,dx=\langle\Lambda_Iq,\varphi|_{\p\Omega}\rangle$. Adding these identities and using \eqref{eq:inner-conductivity-DN} gives $\int_{\Omega_{\mathrm e}}\nabla U\cdot\nabla\varphi\,dx=\langle\Lambda_Iq,\varphi|_{\p\Omega}\rangle-\langle\Lambda_{I+H}q,\varphi|_{\p\Omega}\rangle=0$. The function $U$ is Euclidean harmonic in $\Omega_{\mathrm e}$ and has boundary trace $f$ on $\p\Omega_{\mathrm e}$. Uniqueness for the Dirichlet problem shows that $U$ is the harmonic function with boundary value $f$.
		
		By definition, $U=u$ throughout $\Omega_{\mathrm e}\setminus\ol\Omega$. This region contains a full neighborhood of $\p\Omega_{\mathrm e}$, and $A_{\mathrm e}=I$ there. The outer conormal derivative of $u$ agrees with the normal derivative of $U$ on $\p\Omega_{\mathrm e}$. Since this holds for every boundary value $f$, formula~\eqref{eq:outer-DN-equality} follows.
		
		Finally, we prove the quantitative estimate for the zero extension. Set $\sigma=s_*-k$, so $\sigma=0$ when $n$ is even and $\sigma=1/2$ when $n$ is odd. By \eqref{eq:boundary-flatness}, for every multi-index $\beta$ with $|\beta|\le k$, the distributional derivative $D^\beta H_{\mathrm e}$ is the zero extension of $D^\beta H$. Since $\Omega$ is bounded, the integer-order Sobolev norm satisfies $\|H_{\mathrm e}\|_{H^k(\R^n)}\le C\|H\|_{C^k(\ol\Omega)}$. This proves the required estimate when $\sigma=0$.
		
		Suppose that $\sigma=1/2$, and let $|\beta|=k$. Write $F=D^\beta H_{\mathrm e}$ and $f=D^\beta H$. The field $f$ belongs to $C^\alpha(\ol\Omega)$ and vanishes on $\p\Omega$. Its zero extension satisfies
		\begin{equation}\label{eq:zero-extension-Holder}
			\|F\|_{C^\alpha(\R^n)}\le C\|H\|_{C^{k,\alpha}(\ol\Omega)}.
		\end{equation}
		To verify the seminorm bound across the boundary, let $x\in\Omega$ and $y\notin\Omega$, and choose $p\in\p\Omega$ with $|x-p|=\operatorname{dist}(x,\p\Omega)$. Since $f(p)=0$ and $|x-p|\le|x-y|$, we have $|F(x)-F(y)|=|f(x)-f(p)|\le[f]_{C^\alpha(\ol\Omega)}|x-y|^\alpha$. If both points lie in $\Omega$, the original H\"older estimate applies; if both lie outside $\Omega$, the difference is zero. This proves \eqref{eq:zero-extension-Holder}.
		
		For $0<\sigma<1$, Plancherel's theorem and the change of variables in the Fourier integral give
		\begin{equation}\label{eq:fractional-translation-identity}
			\int_{\R^n}\frac{\|F(\cdot+z)-F\|_{L^2(\R^n)}^2}{|z|^{n+2\sigma}}\,dz=c_{n,\sigma}\int_{\R^n}|\xi|^{2\sigma}|\wh F(\xi)|^2\,d\xi,
		\end{equation}
		where $c_{n,\sigma}>0$. Indeed, the Fourier multiplier of $F(\cdot+z)-F(\cdot)$ is $e^{\mathsf{i}z\cdot\xi}-1$, and integration of its squared modulus against $|z|^{-n-2\sigma}\,dz$ gives a positive constant times $|\xi|^{2\sigma}$.
		
		The support of $F(\cdot+z)-F(\cdot)$ is contained in $(\ol\Omega-z)\cup\ol\Omega$, whose measure is at most $2|\Omega|$. For $|z|\le1$, formula \eqref{eq:zero-extension-Holder} gives $\|F(\cdot+z)-F(\cdot)\|_{L^2}^2\le C\|H\|_{C^{k,\alpha}(\ol\Omega)}^2|z|^{2\alpha}$. For $|z|>1$, translation invariance of the $L^2$ norm gives $\|F(\cdot+z)-F(\cdot)\|_{L^2}^2\le4\|F\|_{L^2}^2\le C\|H\|_{C^{k,\alpha}(\ol\Omega)}^2$. Hence,
		\begin{equation}\label{eq:zero-extension-fractional}
			\begin{aligned}
				\int_{\R^n}\frac{\|F(\cdot+z)-F\|_{L^2}^2}{|z|^{n+2\sigma}}\,dz\le	C\|H\|_{C^{k,\alpha}(\ol\Omega)}^2
				\bigg(\int_0^1 r^{2\alpha-2\sigma-1}\,dr+\int_1^\infty r^{-2\sigma-1}\,dr\bigg)\le C\|H\|_{C^{k,\alpha}(\ol\Omega)}^2,
			\end{aligned}
		\end{equation}
		where the last inequality uses $\alpha>\sigma$.
		
		Applying \eqref{eq:fractional-translation-identity} and \eqref{eq:zero-extension-fractional} to every derivative of order $k$, and using the equivalence of $|\xi|^{2k}$ with $\sum_{|\beta|=k}|\xi^\beta|^2$, gives $\|H_{\mathrm e}\|_{H^{k+\sigma}(\R^n)}\le C\|H\|_{C^{k,\alpha}(\ol\Omega)}$. Together with the case $\sigma=0$ and Proposition~\ref{prop:harmonic-coordinates}, this proves
		\begin{equation}\label{eq:He-Hnplus1-proof}
			\|H_{\mathrm e}\|_{H^{s_*}(\R^n)}\le C\|H\|_{C^{k,\alpha}(\ol\Omega)}\le C\|g-e\|_{C^{k,\alpha}(\ol\Omega)},
		\end{equation}
		where the constants depend only on $n$, $\Omega$, and $\alpha$. This proves \eqref{eq:He-Hnplus1}.
	\end{proof}
	
	\subsection{Proof of the rigidity result}
	
	\begin{proof}[Proof of Theorem~\ref{thm:metric-main}]
		Let $\varepsilon>0$ be the smallness threshold in Proposition~\ref{prop:harmonic-coordinates}. Choose a smooth bounded connected domain $\Omega_{\mathrm e}$ with $\ol\Omega\Subset\Omega_{\mathrm e}$ and a compact set $K_{\mathrm e}$ satisfying $\ol\Omega\subset K_{\mathrm e}\Subset\Omega_{\mathrm e}$. These choices are fixed throughout the proof and depend only on $\Omega$. Let $\varepsilon_0>0$ be the threshold in Theorem~\ref{thm:conductivity-main}, with the domain $\Omega$ and the compact set $K$ in that theorem replaced by $\Omega_{\mathrm e}$ and $K_{\mathrm e}$, respectively, and let $C_{\mathrm e}$ be the constant in \eqref{eq:He-Hnplus1}. Choose $\varepsilon_1>0$ so that $\varepsilon_1\le\varepsilon$ and $C_{\mathrm e}\varepsilon_1<\varepsilon_0$.
		
		Suppose that the assumptions of Theorem~\ref{thm:metric-main} hold. Proposition~\ref{prop:harmonic-coordinates} gives the canonical boundary fixing diffeomorphism $Y$ and the conductivity $A_{\widetilde g}=I+H$ associated with $\widetilde g=Y_*g$. Lemma~\ref{lem:harmonic-boundary-flatness} shows that $H$ vanishes to infinite order at $\p\Omega$. Lemma~\ref{lem:outer-DN-transfer} gives an extension $H_{\mathrm e}$ satisfying $\supp H_{\mathrm e}\subset\ol\Omega\subset K_{\mathrm e}$, $\diver H_{\mathrm e}=0$, and $\Lambda_{I+H_{\mathrm e}}^{\p\Omega_{\mathrm e}}=\Lambda_I^{\p\Omega_{\mathrm e}}$. Using \eqref{eq:metric-smallness}, it also gives
		\begin{equation}\label{eq:metric-Hnplus1-small}
			\|H_{\mathrm e}\|_{H^{s_*}(\R^n)}\le C_{\mathrm e}\|g-e\|_{C^{k,\alpha}(\ol\Omega)}
			<C_{\mathrm e}\varepsilon_1<\varepsilon_0.
		\end{equation}
		Theorem~\ref{thm:conductivity-main} gives $H_{\mathrm e}=0$. In particular, $A_{\widetilde g}=I$ in $\Omega$. Since $A_{\widetilde g}=(\det\widetilde g)^{1/2}\widetilde g^{-1}$, taking determinants gives $\det A_{\widetilde g}=(\det\widetilde g)^{(n-2)/2}$. Thus, for $n\ge3$, $\widetilde g=(\det A_{\widetilde g})^{1/(n-2)}A_{\widetilde g}^{-1}$. Since $A_{\widetilde g}=I$, we obtain $\widetilde g=e$. Finally, $\widetilde g=Y_*g=(Y^{-1})^*g$ implies $g=Y^*e$.
		
		It remains to prove the uniqueness assertion. Let $F\in\Diff^\infty(\ol\Omega)$ satisfy $F|_{\p\Omega}=\Id$ and $g=F^*e$. For each $1\le j\le n$, the function $F^j=x^j\circ F$ satisfies
		$-\Delta_gF^j=(-\Delta_ex^j)\circ F=0$
		and $F^j=x^j$ on $\p\Omega$. The function $y^j$ defining the canonical harmonic map $Y$ solves the same Dirichlet problem. Dirichlet uniqueness gives $F^j=y^j$ for every $1\le j\le n$, so $F=Y$.
	\end{proof}
	
	\begin{corollary}[Rigidity in harmonic coordinates]\label{cor:metric-harmonic-gauge}
		Under the assumptions of Theorem~\ref{thm:metric-main}, suppose in addition that
		\begin{equation}\label{eq:metric-harmonic-gauge}
			\p_i\big((\det g)^{1/2}g^{ij}\big)=0\quad\text{for }1\le j\le n.
		\end{equation}
		Then $g=e$.
	\end{corollary}
	
	\begin{proof}
		Each Cartesian coordinate function $x^j$ is $g$-harmonic under \eqref{eq:metric-harmonic-gauge} and has the same boundary value as $y^j$. Dirichlet uniqueness gives $y^j=x^j$ for every $1\le j\le n$, so $Y=\Id$. The proof of Theorem~\ref{thm:metric-main} gives $Y_*g=e$, and hence $g=e$.
	\end{proof}
	
	\begin{remark}[The canonical rigidity diffeomorphism]\label{rem:diffeomorphism-gauge}
		For a fixed metric $g$, the boundary fixing diffeomorphism in Theorem~\ref{thm:metric-main} is uniquely determined and is equal to the canonical harmonic map $Y$. This does not remove the usual diffeomorphism invariance of the anisotropic Calder\'on problem. Indeed, if $F\in\Diff^\infty(\ol\Omega)$ fixes the boundary and $g_F=F^*e$, then $\Lambda_{g_F}=\Lambda_e$, while the canonical harmonic map associated with $g_F$ is precisely $F$. The boundary data determine the metric only up to a boundary fixing diffeomorphism, whereas the harmonic construction identifies the unique rigidity map after a particular representative $g$ has been fixed.
	\end{remark}

	\section*{Statements and declarations}
	
	\para{Data availability statement} No datasets were generated or analyzed in this study.
	
	\para{Conflict of interest} The author declares no conflict of interest.
	
\para{Acknowledgments} The author is deeply grateful to Mikko Salo for suggesting a direct comparison with the Calder\'on exponential solution, which provided an important starting point for the argument, and for generously reading the manuscript and offering valuable comments.

After completing this work, the author became aware of the preprint by P.~Stefanov \cite{Stefanov2026}, which establishes a closely related rigidity result for the anisotropic Calder\'on problem near the Euclidean metric. The present work was carried out independently.

The author is partially supported by the National Science and Technology Council (NSTC), Taiwan, under the project 113-2115-M-A49-017-MY3. The author also acknowledges financial support from the Alexander von Humboldt Foundation through the Henriette Herz Scouting Programme, hosted by Universit\"at Duisburg-Essen, Germany. The author acknowledges the use of AI tools. All mathematical arguments and proofs in the final manuscript were checked and written by the author.
	\bibliographystyle{alpha}
	\bibliography{refs}
	
\end{document}